\documentclass{article}
\usepackage{amsfonts}
\usepackage{a4wide}
\usepackage{float}
\usepackage[utf8]{inputenc}
\usepackage{natbib}
\usepackage{amsmath,amsthm}
\usepackage{xcolor}
\usepackage{appendix}
\usepackage{graphicx}
\usepackage{hyperref}
\usepackage{algorithm}
\usepackage{subcaption}
\usepackage{multirow}
\usepackage{algpseudocode}

\usepackage{amssymb}
\newtheorem{theorem}{Theorem}
\newtheorem{prop}{Proposition}
\newtheorem{aspt}{Assumption}

\newtheorem{thm}{Theorem}

\newtheorem{lemma}[theorem]{Lemma}

\newcommand{\be}{\begin{equation}}
\newcommand{\ee}{\end{equation}}
\newcommand{\bea}{\begin{eqnarray}}
\newcommand{\eea}{\end{eqnarray}}

\newcommand{\Uhat}{\hat{U}}

\newcommand{\Lhat}{\hat{L}}

\newcommand{\E}{\mathbb{E}}
\newcommand{\reals}{\mathbb{R}}
\newcommand{\Vhat}{\hat{V}}

\newcommand{\Vtilde}{\tilde{V}}

\newcommand{\Dhat}{\hat{D}}

\newcommand{\mhat}{\hat{m}}
\newcommand{\vhat}{\hat{v}}
\newcommand{\fhat}{\hat{f}}
\newcommand{\Ahat}{\hat{A}}
\newcommand{\thetahat}{\hat{\theta}}
\newcommand{\lambdahat}{\hat{\lambda}}

\newcommand{\calC}{\mathcal{C}}
\newcommand{\calChat}{\hat{\calC}}

\newcommand{\Xtilde}{\tilde{X}}

\newcommand{\Etilde}{\tilde{E}}
\newcommand{\Ztilde}{\tilde{Z}}
\newcommand{\cov}{\text{cov}}

\newcommand{\lambdatilde}{\tilde{\lambda}}
\newcommand{\Lambdatilde}{\tilde{\Lambda}}

\newcommand{\Xhat}{\hat{X}}
\newcommand{\supp}{Supplementary Material}
\newcommand{\Lambdahat}{\hat{\Lambda}}
\newcommand{\err}[1]{\|#1\|_{2,\infty}}
\newcommand{\tr}{\text{tr}}

\title{Uniform error bound for PCA matrix denoising}
\author{Xin T. Tong,  Wanjie Wang, and Yuguan Wang}
\date{\today}

\begin{document}

\maketitle 

\begin{abstract}
Principal component analysis (PCA) is a simple and popular tool for processing high-dimensional data. We investigate its effectiveness for matrix denoising. 
We consider the clean data are generated from a low-dimensional subspace, but masked by independent high-dimensional sub-Gaussian noises with standard deviation $\sigma$. 
Under the low-rank assumption on the clean data with a mild spectral gap assumption, we prove that the distance between each pair of PCA-denoised data point and the clean data point is uniformly bounded by $O(\sigma \log n)$. 
To illustrate the spectral gap assumption, we show it can be satisfied when the clean data are independently generated with a non-degenerate covariance matrix. 
We then provide a general lower bound for the error of the denoised data matrix, which indicates PCA denoising gives a uniform error bound that is rate-optimal. Furthermore, we examine how the error bound impacts downstream applications such as clustering and manifold learning. Numerical results validate our theoretical findings and reveal the importance of the uniform error.
% Suppose the observed data points are $Z_i = X_i + \xi_i \in \mathbb{R}^d$, where $X_i$ represents the underlying clean data and $\xi_i \sim N(0, \sigma^2)$ denotes the noise. The noise level $\sigma^2d$ will cause large error for high-dimensional data with large $d$. To denoise this data matrix, principal component analysis (PCA) is a crucial tool, where we projects $Z_i$ to the subspace spanned to the largest several right singular vectors. In this work, we analyze the uniform error between the denoised matrix obtained through PCA and the true clean data when $X$ lies in a low-dimensional space. We establish the upper bound at $\Otilde(\sigma)$ under mild conditions. We further give models to illustrate these mild conditions, such as a zigzag line. We then provide a general lower bound on matrix denoising, which indicates PCA denoising is rate-optimal. Furthermore, we examine how the uniform error bound impacts downstream applications such as empirical risk minimization, clustering, and manifold learning. Numerical results validate our theoretical findings and reveal the importance of the uniform error.
\end{abstract}
\section{Introduction}\label{sec:model}
% Distance is a fundamental concept in many statistical methods. It is can be used to measure the similarity or dissimilarity between  data points, and group data into clusters. Some typical examples include k-means and support vector machine (SVM). 
% Distance is also often used to unveil hidden geometric structure of the data distribution. For example, the diffusion map can be used to obtain a low dimensional manifold where the data is distributed. 

In the modern era, data is often referred to as "the new gold". Rich data with rapidly increasing statistical methods present us with powerful tools for extracting valuable information and explaining scientific problems. 
 However, the process of collecting data inevitably introduces noise, which poses a significant challenge. While statistical methods typically exhibit stability in the presence of weak noise, they may struggle to perform well when the noise surpasses the signal present in clean data. This issue becomes particularly pronounced in the realm of high-dimensional data where each dimension of the data point is corrupted by noise. As the number of dimension grows, the overall noise also grows, which further exacerbates the curse of dimensionality  \citep{donoho2000high} when we try to analyze high-dimensional data given a small number of observations.

Consider $n$ observed data points $Z_i = X_i + \xi_i \in \reals^d$, where $X_i$ is the underlying truth and $\xi_i$ is an independent sub-Gaussian noise vector with covariance $\sigma^2 I$, $i \in [n]$. $X_i$, also called a {\it clean data point}, is unknown to us and algorithms are conducted on $Z_i$s. While the observed data $Z_i$ usually fall into a high-dimensional space, we assume the intrinsic dimension of the clean data is low, i.e. it can be embedded into a low-dimensional subspace. 
It is common to assume that $\|X_i\|$'s are bounded by a constant; see \cite{von2008consistency}. Without loss of generality, we assume $\|X_i\| \leq 1$. In practice, this can be achieved by data normalization.
%Take a case where $n$ clean data points $X_i\in\reals^{d}$ are corrupted by independent and identically distributed (i.i.d.) Gaussian noise $\xi_i\sim \mathcal{N}(0, \sigma^2 I_d)$. In other words, what we can observe are the noisy data points $Z_i=X_i+\xi_i$, $i \in [n]$. For simplicity of exposition, we assume $\|X_i\|\leq 1$. \blue{This assumption helps us to interpret the relative errors and it can be met in practice by normalizing the data}. 

To evaluate the data quality of $Z_i$, we introduce the signal-to-noise ratio (SNR). SNR is a common metric that measures the relative strength of the signal when compared with the noise. In this simple setting, the SNR is given by: 
 \[
\text{SNR}=\max_{i\in [n]}\frac{\|X_i\|^2}{\E[\|\xi_i\|]^2}\leq \frac{1}{\sigma^2 d}.
 \] 
 When the dimension $d$ increases, the SNR deteriorates and tends towards zero. With a low SNR, analyzing data directly based on $Z_i$ will induce unsatisfactory results. Naturally, we seek to denoise $Z_i$ first to improve the accuracy of data analysis. This procedure is known as matrix denoising in the literature (see for example \citep{donoho2014minimax}), and we introduce it in the following section.
% Then for each pair of clean data points $X_i$ and $X_j$, their distance $\|X_i-X_j\|$ will be a number between $0$ and $2$. However, the distance between their observed noisy counterparts is 
% \[
% \|Y_i-Y_j\|=\|X_i-X_j+\xi_i-\xi_j\|,\quad \text{where }\|\xi_i-\xi_j\|=O(\sqrt{2d}\sigma)
% \]
% with high probability. So if $\sigma$ is a fixed $O(1)$ constant, and the dimension $d$ is sufficiently large,   the noisy distance  $\|Y_i-Y_j\|$ is mainly contributed by the noise realization. Alternatively, 

\subsection{PCA for denoising}

Suppose the clean data points $X_i$ are distributed in a low-dimensional subspace with a dimension of $r$, where $r \ll d$. A direct idea to recover $X_i$ is to use the singular value decomposition (SVD) of the noisy data matrix formed by $Z_i$. 
%The SVD will accumulate the signals across all observations without cancellation, which is comparable with the noise matrix. 

We introduce the matrix form for the data points. Let $Z = [Z_1, \ldots, Z_n]^T$, $X = [X_1, \ldots, X_n]^T$ and $E = [\xi_1, \ldots, \xi_n]^T$ be the matrices formed by the noisy data points, clean data points, and pure noise, respectively. 
%Let $X=[X_1,...., X_n]^T\in \reals^{n\times d}$, where each row represents a clean data point. Similarly, we have the noisy data $Z=[Z_1,\ldots, Z_n]^T=X+E$, where the rows of $E$ are $\xi_i \sim \mathcal{N}(0, \sigma^2 I_d)$.  
According to the assumption, $rank(X) = r$. We start from the following hypothetical denoising procedure and then propose the algorithm.
%let us consider the following hypothetical procedure. 
Let the SVD of $X$ be denoted as $X=U\Lambda V^T$, where $U\in \mathbb{R}^{n\times r}$, $\Lambda\in \mathbb{R}^{r\times r}$, and $V\in \mathbb{R}^{d\times r}$. The columns of $V$ span the $r$-dimensional subspace in which $X$ lies, 
%with a dimension of $r$ 
and $VV^T$ is the projection operator onto this subspace. By applying this projection operator to $Z_i$, we obtain:
\[
VV^T Z_i=X_i+ VV^T\xi_i,\quad \text{where }\E[\|VV^T\xi_i\|^2]=\sigma^2\text{tr}(VV^T)=r\sigma^2.
\]
Thus, when the underlying dimension $r = O(1)$ and $r \ll d$, %when the dimension of the subspace, $r$, is fixed and much smaller than $d$, 
the noise in the projection $VV^T Z_i$ is significantly weaker than the noise in $Z_i$. 
%For the projected data $VV^T Z_i$, 
the associated SNR of the projected data $VV^TZ_i$ is of order ${1}/{(\sigma^2r)}$, which can be sufficiently strong to yield accurate inference results. Therefore, by leveraging the SVD and performing the projection onto the low-dimensional subspace, we can effectively denoise the data and obtain accurate estimates.

In practice, there is no access to $V$. Therefore, we estimate it using the SVD of the noisy data matrix $Z$. The SVD of $Z$ can be expressed as $Z=\Uhat \Lambdahat \Vhat^T$, where typically $Z$ has full rank due to the presence of noise. To focus on the most significant components of the data, we select only the first $r$ columns of $\Vhat$, denoted as $\Vhat_r$, corresponding to the largest $r$ singular values. 
By $\Vhat_r$, we project the noisy data points $Z_i$ onto the estimated subspace, resulting in the denoised estimates $\Xhat_i$. The projection is given by:
\[
\Xhat_i=\Vhat_r\Vhat_r^T Z_i,\quad i\in [n].  
\]
The columns of $\Vhat_r$ can be interpreted as the $r$ directions that capture the most variability of the data points $Z_i$. Therefore, they are often referred to as the principal directions, and the resulting $\Xhat_i$ are known as the principal components. This approach is commonly known as principal component analysis (PCA). 
We call it the PCA-denoising algorithm,  presented in Table \ref{alg:denoise}. 

Utilizing PCA for noise reduction is not a new concept. It was first introduced in multivariate statistical analysis and then explored in various fields. 
For example, 
\cite{shepard1962analysis} introduced the use of PCA for multidimensional scaling and distance estimation. 
In the field of image processing, \cite{singh1985standardized} applied PCA to denoise images. Discussions in Section \ref{review} provide more details of the related literature and results. We also refer interested readers to surveys and textbooks for more comprehensive lists \citep{jolliffe2005principal,abdi2010principal,chen2021spectral}. 

The denoised data $\Xhat$ can be applied to various applications, such as empirical risk minimization, clustering, manifold learning and so on. The denoising step largely improves the performance of algorithms in these fields. More discussions can be found in our Section \ref{sec:application}. 

\begin{algorithm}
\caption{PCA-denoising}\label{alg:denoise}
\begin{algorithmic}[1]
\Require Data  $Z\in \reals^{n\times d}$, dimension $r$. 
\Ensure Denoised data $\Xhat\in \reals^{n\times d}$.
\State Find the SVD of $Z$ as $Z=\Uhat \Lambdahat \Vhat^T$.
\item Denote the first $r$ columns of $\Vhat$ as $\Vhat_r \in \reals^{d \times r}$. 
\State Let $\Xhat=Z\Vhat_r\Vhat_r^T$. 
\end{algorithmic}
\end{algorithm}

% $R$ is the dimension of the subspace that $X$ lives in. Then $V$ spaces the subspace that $X$ lives in. And if we project $X$ on to the subspace, the result will be $XV=U\Lambda$. Notably, such projection is an isometry, i.e. $\|X_i-X_j\|=\|(e_i-e_j)^T U\Lambda\|$. So applying graphical Laplacian to $X$  and $U\Lambda$ will yield the same eigenvectors, while the latter is of low $(R)$ dimension, so it is more stable against noise.

% In high-dimensional data, noise can be a significant issue because the amount of data is large, and the noise can be spread across many dimensions. This can make it challenging to identify meaningful patterns or relationships in the data, and can also make it difficult to develop accurate predictive models.

\subsection{Our main interest and contribution}
One crucial question for matrix denoising is assessing the accuracy of $\Xhat_i$, i.e. the distance between the estimate $\Xhat_i$ and the clean data $X_i$ for all $i \in [n]$. 
Most existing theoretical analysis of PCA focuses on the Frobenius distance between the two matrices $\Xhat$ and $X$,
which is the squared root of the sum of squared $\ell_2$ error among all data points.
%\blue{which is the average $l_2$ error square among all data points}. 
Our goal is to obtain a uniform error bound across all data points, which allows for individual statistical analysis on each sample. 
Specifically, we aim to establish the following $\ell_2\to \ell_\infty$ or uniform error bound for the PCA-based denoising algorithm,
\begin{equation}
\label{eqn:uniform}
\|\hat{X} - X\|_{2,\infty} := \max_{a \in \reals^d, a \neq 0}\frac{\|(\hat{X} - X)a\|_{\infty}}{\|a\|_2}=\max_{i\in [n]} \|X_i-\Xhat_i\|=O(\sigma\log n)
\end{equation}
with high probability, when $d = c n$ with $c$ and $\sigma$ being some absolute constants.
%\blue{when $d=c_dn$} with $c_d$ and $\sigma$ are some absolute constants. 
Here the notation $O(\cdot)$  hides a factor that may depend on the low dimension $r$.
% Here $\Otilde(\sigma) = O(\sigma \log (\blue{n\vee d})^c)$ for some \blue{abosolute constant} $c > 0$. 
The equivalence between $\err{\Xhat-X}$ and $\max_{i\in [n]}\|X_i-\Xhat_i\|$ can  be found in \cite{twoinifinity1, twoinifinity2}, which provides further theoretical insights into our results.

To understand our goal \eqref{eqn:uniform}, we consider a special low-dimensional case where the observed data $Z_i$ also has a dimension of $r$, i.e., $d = r$. In this case, the noise $E_i$ has a low dimension of $r$ and $\|E\| = O(\sigma^2r)$
is small. Therefore, the dimension reduction is not necessary and we can directly use the observed data $Z_i$ as an estimate of the clean data $X_i$. By a union bound, we can find the estimation error is  $\|Z-X\|_{2,\infty} = \|E\|_{2,\infty} =O(\sigma\log n)$, the same as \eqref{eqn:uniform}. Therefore, our goal \eqref{eqn:uniform} implies that the PCA-denoising estimates $\Xhat_i$ achieve the same level of accuracy as in the low-dimensional case.
In other words, the PCA-denoising step essentially removes the curse of dimensionality.

In this paper, we explore the estimate \eqref{eqn:uniform} from several perspectives. Here is a summary of our main findings and the organisation of this paper:
\begin{enumerate}
    \item In Section \ref{sec:unif}, Theorem \ref{thm:linftydavis} establishes a general form of \eqref{eqn:uniform} for arbitrary $d$ and $\sigma$ under the condition $\|X_i\| \leq 1$. When assumptions $n\asymp d$ and the $r$-th largest singular value of $X$ that $\lambda_r(X) \geq c_X \sqrt{n}$ holds for an absolute constant $c_X$, we show that the estimate \eqref{eqn:uniform} holds exactly. 
    Our result does not impose any assumptions on the correlation structure of clean data $X_i$.
    %The results hold when $X_i$ are non-i.i.d. and have an arbitrary correlation.  
    \item In Section \ref{sec:eigenvalue}, we investigate the sufficient conditions that the assumption $\lambda_r(X)\geq c_X\sqrt{n}$ holds. By the random matrix theory, we demonstrate that the covariance matrix of $X$ with a non-zero $r$-th eigenvalue will suffice. 
    %As an example, we use it to show that $X$ distributed on a zigzag line will meet this assumption. 
    We illustrate this assumption by a zigzag line example, which is motivated by temporal and spatial data sets.
    \item Section \ref{sec:lowerbound} presents a general lower bound on the signal-to-noise ratio and sample size $n$ to ensure that the average error is no larger than any constant $\epsilon > 0$. The lower bound highlights that PCA-denoising has the rate-optimal signal-to-noise ratio and sample size requirement.

    \item In Section \ref{sec:application}, we demonstrate the practical implications of the uniform error bound $\err{\Xhat - X}$ in various downstream applications. Assuming $\err{\Xhat - X} \leq \epsilon$, we provide performance guarantees for applications such as clustering and manifold learning.
    \item Finally, in Section \ref{sec:simulation}, we provide some numerical simulations to support our theoretical findings. We consider a clustering task on high-dimensional data sampled from two separated zigzag lines. 
    For this data, PCA-denoising yields the uniform error bound \eqref{eqn:uniform}, and the denoised data enables efficient spectral clustering. We also show that data with a small "average error" is not sufficient to guarantee good clustering results for every sample in this task. 
    \end{enumerate}

% At last, we want to point out that the uniform error is not a new term. 
% In the statistical literature, it is sometimes referred to as the $2\to\infty$ norm, which is defined by 
% \[
% \|\hat{X} - X\|_{2,\infty} = \max_{a \in \reals^d, a \neq 0}\frac{\|(\hat{X} - X)a\|_{\infty}}{\|a\|_2}. 
% \]
%  We call it "the uniform error" throughout this paper to emphasize the intuitive understanding of the concept and its implications.

\subsection{Related literature}\label{review}
Our findings reside at the intersection of PCA and matrix denoising, where plenty of related results exist in the literature. In this section, we will provide a brief overview of the relevant literature from the perspectives of PCA and matrix denoising. For ease of discussion, we would assume $n\asymp d$ in this section, which is a common setting for many works in the discussion. 

\subsubsection{Comparing with PCA literature}\label{subsec:PCA}
Due to the wide range of applications for PCA, there is numerous of literature on its design and applications \citep{abdi2010principal}.
The earliest works can be traced back to the 1960s \citep{rao1964use,jolliffe1972discarding}, where the discussions focus on multivariate statistical analysis. 
However, a rigorous understanding of PCA in high-dimensional settings emerged much later, mostly in the last 15 years. 
In the theoretical analysis of PCA, most studies have focused on the accuracy of subspace recovery, 
i.e. $\|\Vhat_r\Vhat^T_r -VV^T\|_*$. 
Here $\|\,\cdot\,\|_*$ is some operator norm, where $\ell_2$-operator norm is used in most classical settings. 
Denote $\mathcal{E}_V = \Vhat_r\Vhat^T_r -VV^T$ for short.
%, and similarly $\mathcal{E}_U$ can be defined. 
Using the eigenvector perturbation results like \cite{daviskahan,wedin1972perturbation,stewart1990matrix,bhatia2013matrix,yu2015useful}, and random matrix theory, 
an upper bound of the form $\|\mathcal{E}_V\|^2 \leq O(\sigma^2 )$ can often be obtained, with additional conditions on $X$ \citep{vu2011singular,chen2021spectral}.   

In particular, \cite{johnstone2009consistency} discussed the consistency of PCA recovery when the underlying dimension $r=1$ and the single principal component is sparse. Assuming the existence of a spectral gap, \cite{cai2018rate} discussed the optimality of PCA recovery, in the context of the $\ell_2$ operator norm $\|\mathcal{E}_V\|$ and Frobenius norm $\|\mathcal{E}_V\|_F$. It is worth mentioning that $r$ can grow as $O(n)$ in this work. In \cite{zhang2022heteroskedastic}, the results were extended to the setting where the data distribution is heteroskedastic. Other than $\mathcal{E}_V$, $V_r$ is also discussed in the literature. \cite{ding2020high} studied the limiting distribution of $\Vhat_r$. 
\cite{fan2018eigenvector} and \cite{abbe2020entrywise} have shown the  $\ell_\infty$ error of estimated eigenvectors in $\Vhat_r$ is $O(1/n)$, assuming the true eigenvectors in $V$ have $\ell_\infty$ norm being $O(1/\sqrt{n})$.   
\cite{twoinifinity2} studied the distribution of $\Vhat_r-VO$, where $O$ is an orthogonal matrix, and establishes an error bound in $\err{\cdot}$ norm.  In \cite{AveError}, the error is studied when $VV^T$ is replaced by a rank-$r$ projection $P$ that minimizes $\mathbb{E}\|Z - PZ\|^2$.  

The model setting of this paper follows a similar line as these works. In particular, we assume the existence of a spectral gap of $X$, which enables the spectral perturbation analysis. Such a spectral gap is necessary in the existing works discussed above. One of the key steps in the proof, Proposition \ref{prop:leave1}, can be interpreted as an application of the leave-one-out method discussed in \cite{chen2021spectral}. The difference between this paper and these works mainly lies in the errors of interest, as explained below.

Using the upper bound on $\|\mathcal{E}_V\|$, we can straightforwardly obtain the denoising accuracy of PCA for a new data point $Z_{n+1} = x + \xi_{n+1}$, which is often referred as the \emph{test error} in statistics \citep{hastie2009elements}. In particular, we can define and calculate
\begin{align*}
\text{Test}_{\Vhat}&:=\max_{x\in \text{span}(V),\|x\|=1} \E_{\xi}[\|\Vhat_r\Vhat_r^T(x+\xi_{n+1})-x\|^2]\\
&=\max_{x\in \text{span}(V),\|x\|=1} \E_{\xi}[\|\mathcal{E}_Vx\|^2+\|\Vhat_r\Vhat_r^T\xi_{n+1}\|^2]=O(\sigma^2).
\end{align*}
This noise level is the same as our goal in \eqref{eqn:uniform}, if the $\log n$ factor is disregarded.
However, in the proof, the independence between $\xi_{n+1}$ and $\Vhat_r$ is required to obtain the bound on $\|\Vhat_r\Vhat^T_r\xi_{n+1}\|^2$. Without this independence, such as in the calculation of  $\|\Vhat_r\Vhat^T_r\xi_k\|^2$ for $k \in [n]$,  this bound does not hold.  In other words, bounding $\|\mathcal{E}_V\|$ does not directly lead to the desired bound \eqref{eqn:uniform}.

While Test$_{\Vhat}$ is already useful in many situations, it imposes limitations in practical applications. To ensure the independence condition between $\Vhat_r$ and the data points to be projected, one has to split the samples into two sets: one training data set to obtain the projection $\Vhat_r\Vhat_r^T$,  and another set where this projection is applied. Test$_{\Vhat}$ cannot be applied if the projection is on the training data set. This data-splitting approach leads to two problems: 1) It reduces the sample size, resulting in a loss of estimation accuracy, which is undesirable when the original dataset has a limited number of samples. 2) In many unsupervised learning tasks, statistical inference on the training data set itself is crucial. For example, we want to classify all data points in the clustering problem.  
Yet the clustering error on the training data cannot be evaluated using Test$_{\Vhat}$. 

In contrast, a uniform bound of the form \eqref{eqn:uniform} guarantees the result without data splitting. We can obtain $\Vhat_r$ from all samples and apply the projection to all samples, with a uniform error bound. It is more accurate with a larger sample size and allows us to carry out unsupervised operations on all the samples.

\subsubsection{Comparing with matrix denoising literature}
In the matrix denoising literature, the main interest is to find an estimate $\Xhat(Z)$ so that the Frobenius norm $\|\Xhat(Z) - X\|_F$ can be well bounded. Various approaches have been introduced to tackle this problem. \cite{donoho2014minimax} considered using the minimizer of a regularized loss function $\|Z-\Xhat\|^2_F+\lambda \|\Xhat\|_*$, where $\|\Xhat\|_*$ is the nuclear norm of $\Xhat$. The approach has been proved to have a rate-optimal mean squared error in the Frobenius norm. The shrinkage method instead considers estimators of the form $\Xhat=\Uhat g(\Lambdahat)\Vhat^T$, where $g(\Lambdahat)$ is a diagonal matrix with diagonal entries being $g(\lambdahat_i)$. The main interest is to find a good choice of $g$. Assuming the distribution of $X$ is known, \cite{nadakuditi2014optshrink} provided the optimal $g$ for the Frobenius norm. When $X$ is a spiked covariance matrix, most eigenvalues of the estimator $\hat{X}$ are 1.  
Focusing on the spectrum of $\Xhat - X$, which has very few non-zeros,  
\cite{donoho2018optimal} provided the optimal $g$ under the $\ell_p$ norm of the spectrum for any general $p$.
%which can be the $\ell_p$ norm of the error matrix spectrum. 
\cite{leeb2021matrix} extended this discussion to more applied scenarios, where some data can be missing or the singular vectors are sparse.  
Under the assumption that $X$ is low-rank and rotation invariant (or the singular vectors are sparse), \cite{ding2020high} 
 designed a stepwise SVD algorithm that could estimate the singular vectors by going through the data sequentially.  \cite{montanari2018adapting}  discussed the scenario where the noise is from an unknown distribution and suggested a kernelized estimator to do shrinkage.

Compared with these methods, the PCA-denoising approach can be seen as a simple shrinkage method, with $g(\lambda_k)=1_{k\leq r}\lambda_k$. It is much simpler and widely adopted in practice, with few requirements on the clean data $X$. Therefore, the theoretical analysis of it is worth special interest. For example,  \cite{cai2018rate} has discussed the implication of PCA subspace recovery on matrix denoising, and we compared with it in Section \ref{subsec:PCA}.

All existing analyses of matrix denoising focus on bounding the error in the Frobenius norm (or $\ell_p$ generalization of it; see \cite{donoho2018optimal}) which can also be seen as the average mean square error over all data points:
\begin{equation}
\label{eqn:avgerror}
    \frac{1}{n}\sum_{i=1}^n\|X_i-\Xhat_i\|^2=\frac{1}{n}\|X-\Xhat\|_F^2=O(\sigma^2).
\end{equation}
This average error is at the same order as the uniform error bound in \eqref{eqn:uniform} up to a logarithm term. 
But the error bound of form $\|\Xhat-X\|_{2,\infty}$ provides a stronger mathematical guarantee as it captures the distribution of errors across all data points, while the error bound in Frobenius norm allows for some outliers with large individual errors. 
This distinction becomes crucial when analyzing the estimates on individual data points and nonlinear statistical models. 

For instance, in the context of clustering, using the Frobenius error bound would only allow us to establish that the proportion of incorrectly classified data points tends to zero, i.e., the error rate goes to zero. However, it does not provide insights into the exact number of errors, which is often referred to as "strong consistency" in recent statistical literature \citep{abbe2015community, fan2018eigenvector, hu2023network}. On the other hand, a uniform error bound enables us to establish such strong consistency results, as demonstrated by Corollary \ref{cor:cluster} in Section \ref{sec:clustering}.

Such outliers with big errors may also cause challenges for the statistical models with nonlinear and local dependency on the data input, such as K-nearest-neighbor and neural networks. In Section \ref{sec:manifold}, we demonstrate the error of manifold learning using our new uniform error bound, which extends the results in \cite{von2008consistency} to the settings where data is corrupted by high-dimensional noise. In Section \ref{sec:simulation}, we provide numerical simulations on graphical Laplacian spectral clustering. There, we show that PCA-denoising leads to stable clustering results, while data with the same Frobenius error fails to guarantee satisfactory clustering outcomes. 

The analysis technique of this paper differs from these works on matrix denoising, as it is closer to the techniques used in PCA analyses. This is reflected by the assumptions we made for $X$. In particular, our result requires the $r$-th singular value of $X$ to scale like $\sqrt{n}$. Such spectral gap requirement is not needed in \cite{donoho2014minimax} and \cite{donoho2018optimal}. \cite{nadakuditi2014optshrink} and \cite{bao2021singular} assumed the existence of a spectral gap, but its characterization is sharper than the one used in this paper. On the other hand, such an assumption allows this paper and \cite{cai2018rate} to discuss general cases where $d$ and $n$ diverge at different speeds, while most matrix denoising works focus on the scenario $n\asymp d$.

\subsection{Notations}
% For a random variable $Z \sim F$, we use $\|Z\|_{\psi_2}$ to denote the sub-gaussian norm of $Z$, where $P(|Z| \geq t) \leq \exp(-t^2/\|Z\|_{\psi_2}^2)$
For any matrix $A$, we use $A_i$ to denote the $i$-th row vector of $A$ and $A_{i,j}$ to denote the $(i,j)$-th entry. Denote $\lambda_k(A)$ the $k$-th largest singular value of $A$. Denote $I_k$ to be the $k \times k$ identity matrix. We denote the $\ell_2$-operator norm of $A$ as $\|A\|$ and the Frobenius norm of $A$ as $\|A\|_F$. 

To generalize our discussion to other noise types, we introduce the notion of sub-Gaussian random variables. We say a random variable  $Z$ is sub-Gaussian$(0,\sigma^2,K_{\psi_2})$ if
\[
\E[Z]=0,\quad \E[Z^2]=\sigma^2,\quad K_{\psi_2} \geq \sup\nolimits_{p\geq 1}\E[|Z/\sigma|^p]^{1/p}. 
\]
Given a constant $C$, we say $C$ is absolute if it does not depend on any other constants. We write $C=C(x,y)$ if it only depends on constants $x$ and $y$.
In this paper, without further statements, we fix the sub-Gaussian parameter $K_{\psi_2}$ and the rank $r$ as absolute constants.
For two series $a_n$ and $b_n$, we say $a_n \lesssim b_n$ or $a_n=O(b_n)$ if there is a constant $C=C(r,K_{\psi_2})$ so that $\limsup_{n \to \infty} a_n/b_n \leq C$.
Similarly, we have $a_n \gtrsim b_n$.
We say $a_n \asymp b_n$ if there is a constant $C$, such that $a_n \leq C b_n$ and $b_n \leq C a_n$ when $n$ is large enough.  Finally, we use the notation $[N]:=\{1,\ldots,N\}$ for any integer $N$.

\section{Performance bounds for PCA denoising}

\subsection{Uniform bounds for PCA-denoising}\label{sec:unif}
In this section, we establish the upper bound for the uniform error $\err{\Xhat - X}$, where $\Xhat$ is obtained from Algorithm \ref{alg:denoise}. For notational simplicity, we write $\lambda_r = \lambda_r(X)$, which is the $r$-th largest singular value of $X$. 
%Our results demonstrate the error is of order $O(n\sigma)$ when $\sigma^2n \leq \lambda_r^2$.
% consider general values of $\lambda_r$ and $\sigma$, and then provide a simplified version when $\lambda_r \geq c_X\sqrt{n}$ for a constant $c_X > 0$.

%It can be further simplified to $O(\sigma)$ when $\lambda_r > c_X\sqrt{n}$, which is our goal \eqref{eqn:uniform}.
\begin{thm}
\label{thm:linftydavis}
 Suppose $X\in \reals^{n\times d}$ has rank $r$, where each row  $X^T_i$ is bounded by $\|X_i\|\leq 1$. Suppose $E\in \reals^{n\times d}$ has independent entries being sub-Gaussian$(0,\sigma^2,K_{\psi_2})$ distributed with $\sigma\leq 1$. Let $\Xhat$ be the denoised data matrix from Algorithm \ref{alg:denoise}. 
Then there are constants $n_0, C_1=C_1(K_{\psi_2})$ and $C_2= C_2(K_{\psi_2}, r)$, so that if
\begin{equation}
\label{eqn:specgap}
n>n_0,\quad \lambda_r>1+C_1\sigma (\sqrt{n}+\sqrt{d}),    
\end{equation}
 the following holds with probability $1 - O(1/n)$, 
\begin{align*}
\err{\Xhat - X} &\leq
C_2\min\left\{\frac{\sigma(\sqrt{n}+\sqrt{d})}{\lambda_r}+\frac{n\sigma \sqrt{\log n}}{\lambda_r^2},\frac{\sigma}{\lambda_r^2} (\sqrt{nd}+n\sqrt{\log n})\right\}\\
&\quad +\frac{C_2\sigma}{\lambda_r^4}\sqrt{\log n}(n+\sigma^2n\sqrt{n}+\sigma^2 d\sqrt{n})(1+\sigma\sqrt{\log n})(1+\sigma(\sqrt{d}+\sqrt{\log n})).
\end{align*}
Further, if  $\lambda_r \geq c_X\sqrt{n}$ holds for an absolute constant $c_X > 0$, then the bound can be shortened as below for another $C_2=C_2(K_{\psi_2},r, c_X)$ 
\[
\err{\Xhat - X} \leq \begin{cases} C_2\cdot\left(\sqrt{\frac{d}{n}}\sigma(1+\frac{d}{n}\sigma^3\sqrt{\log n}+\frac{d}{n}\sigma^4 \log n)+\sqrt{\log n}\sigma\right),\quad &d\gtrsim n;\\
C_2\sqrt{\log n}\sigma(1+\sigma^4 \sqrt{\frac{d\log n}{n}}),\quad &d\lesssim n.
\end{cases}
\]
If we further assume $c_d=d/n$ and  $\sigma$ are absolute constants, then the bound can be shortened as below for another $C_2=C_2(K_{\psi_2},r, c_d, \sigma)$
\[
\err{\Xhat - X} \leq C_2\sigma \log n. 
\]
\end{thm}
Theorem \ref{thm:linftydavis} presents three estimation error bounds because of the trade-off between formula complexity and generality of settings, the first estimate being under the most general setting and the last estimate being under a strict but canonical setting. 
The first error bound allows the key constants $n,\sigma$, and $d$ to be free from each other, which covers most settings. For example, it can cover the case $d \gg n$ in \cite{cai2018rate} and the case $d\sigma\asymp 1$ in \cite{montanari2018adapting}. The error bound can be simplified under specific settings. 
The second error bound is a simplification when $\lambda_r \geq c_X\sqrt{n}$. Compared to the requirement $\lambda_r > 1 + C_1 \sigma(\sqrt{n} + \sqrt{d})$, this is a more strict condition since it avoids the case that $\sigma$ decays when $n$ increases. 
The last estimate achieves our goal in \eqref{eqn:uniform}. To achieve it, we assume $d\asymp n$ and $\sigma\asymp 1$, which is a canonical setting and can be found in for example \cite{donoho2014minimax} and \cite{abbe2020entrywise}.

Theorem \ref{thm:linftydavis} has two assumptions to achieve the uniform bound on the error. The first assumption is the clean data $X$, that $\max_{i\in [n]}\|X_i\|\leq 1$.
Such a condition is common in manifold analysis, such as \cite{von2008consistency}. 
It can be achieved in any data set $X$ by dividing each data point by a large constant $C \geq \max_{i\in [i]}\|X_i\|$. This assumption also makes the discussed error the relative error. The second assumption is the spectral gap \eqref{eqn:specgap}. It is an essential requirement that the signal level in $X$ should be no smaller than the noise level in $E$. 
% \wanjie{I cannot understand this sentence. Try to rewrite or remove: If we assume $\lambda_r=c_X\sqrt{n}$ and $ d\gtrsim n$, \eqref{eqn:specgap} is essentially a requirement that the SNR is no smaller than $\frac{1}{\sigma^2 d}$.}
 Theorem \ref{thm:lowerbound} below will show the necessity of such a requirement. 

Before we provide the formal proof, it might be worth doing a naive one using $\|\Vhat_r\Vhat_r^T-VV^T\|=O(\sigma)$ from the standard Davis--Kahan Theorem. One may attempt to approach the following bound (the authors tried this in the beginning)
\begin{align*}
\|X_i-\Xhat_i\|&=\|VV^TX_i-(\Vhat_r\Vhat^T_r)(X_i+\xi_i)\|\\
&\leq \|(VV^T-\Vhat_r\Vhat^T_r)X_i\|+\|\Vhat_r\Vhat^T_r\xi_i\|=\|\Vhat_r\Vhat^T_r\xi_i\|+ O(\sigma).
\end{align*}
The main issue comes from the term $\|\Vhat_r\Vhat^T_r\xi_i\|$. A simple bound that $\|\Vhat_r\Vhat^T_r\xi_i\|\leq \|\xi_i\|=O(\sigma\sqrt{d})$ is too loose when $d \to \infty$. 
One may notice that the true projection $VV^T$ causes $\|V V^T\xi_i\|\approx \sqrt{\tr(VV^T)}=\sigma\sqrt{r}$ and expect a similar bound for $\|\Vhat_r\Vhat^T_r\xi_i\|$. However, obtaining such a bound is nontrivial because $\xi_i$ is dependent on $\Vhat_r$. This is exactly where the mathematical challenge lies. We introduce a proposition in Section \ref{sec:loo} to solve this problem by the leave-one-out trick.

\begin{proof}
Note that $U\Lambda O=XVO$ holds for any orthonormal matrix $O\in \reals^{d\times d}$. Recall that $\|X_i\|\leq 1, \|\Vhat_r\|\leq 1, \|V\|\leq 1$. Therefore, we have  

\begin{align}\label{tmp:decomp}
\|X_i-\Xhat_i\|&=\|VV^TX_i-(\Vhat_r\Vhat^T_r)(X_i+\xi_i)\|\nonumber\\
&\leq \|\Vhat_r\Vhat^T_r\xi_i\|+\|(VV^T-\Vhat_r\Vhat^T_r)X_i\|\nonumber\\
&\leq \|\Vhat^T_r\xi_i\|+\|(VO-\Vhat_r)\Vhat^T_rX_i\|+
\|VO(O^TV^T-\Vhat^T_r)X_i\|\nonumber\\
&\leq \underbrace{\|\Vhat^T_r\xi_i\|}_{(a)}+2\underbrace{\|VO-\Vhat_r\|}_{(b)}. 
\end{align}
%The second term can be bounded by $\|\Vhat_r - VO\|$. 
% Now we consider decomposing the first term. 
% \begin{align*}
% \Uhat_r\Lambdahat_r-U\Lambda O&=\Xhat \Vhat_r-XVO=
% (X+E)\Vhat_r-XVO=E\Vhat_r+X (\Vhat_r-VO).
% \end{align*}
% Introduce it into \eqref{eqn:xhatstep1} and recall that $E^T e_i  =\xi_i, X^T e_i  =X_i$.   
% So we have 
% \begin{equation}
% \label{tmp:decomp}
% \|\Xhat_i-X_i\|\leq 
%  \underbrace{\|\xi_i^T\Vhat_r\|}_{(a)}+\underbrace{\|X_i^T (\Vhat_r-VO)\|}_{(b)}+\underbrace{\|\Vhat_r-VO\|\|\Lambda U^Te_i\|}_{(c)}. 
% \end{equation}
%We will discuss each part of the decomposition. We begin with parts (b) and (c) by the Davis--Kahan Theorem, and then the most challenging part (a). 

We will discuss each part of the decomposition. We begin with part (b) by the Davis--Kahan Theorem, and then the much more challenging part (a). 

\textbf{Part (b).} 
Consider the self-adjoint extension $\mathcal{E}(X):=\begin{bmatrix}
0, &X\\
X^T, &0
\end{bmatrix}$ of $X$. First note that  if $X=U\Lambda V^T$ is the SVD decomposition, $\mathcal{E}(X)$ has eigenvalue decomposition of form
\begin{equation}
\label{eq:ext}
\mathcal{E}(X)=\begin{bmatrix}
0, &X\\
X^T, &0
\end{bmatrix}=\frac{1}{\sqrt{2}}\begin{bmatrix}
U, &-U\\
V, &V
\end{bmatrix}\begin{bmatrix}
\Lambda, &0\\
0, &-\Lambda
\end{bmatrix}\frac{1}{\sqrt{2}}\begin{bmatrix}
U^T, &V^T\\
-U^T, &V^T
\end{bmatrix}.
\end{equation}
Since $E$ has independent sub-Gaussian rows, Theorem 5.39 of \cite{vershynin} indicates that there is a constant $C_1=C_1(K_{\psi_2})$ so that with probability $1-O(1/n^2)$,  
\[
\|\mathcal{E}(X+E)-\mathcal{E}(X)\|=\|\mathcal{E}(E)\|\leq 2\|E\|\leq C_1\sigma(\sqrt{d} + \sqrt{n}).
\]
Therefore, by the Davis--Kahan Theorem \citep{daviskahan}, there exists an orthogonal matrix $O$ and constant $C_2=C_2(K_{\psi_2},r)$, so that 
\begin{equation}\label{eqn:vsintheta}
    \|\Vhat_r - VO\|\leq
    \left\|\begin{bmatrix}\Uhat_r\\ \Vhat_r\end{bmatrix}  - \begin{bmatrix}UO\\ VO\end{bmatrix}\right\|
    \leq \frac{C_2\sigma(\sqrt{d}+\sqrt{n})}{\lambda_r}.
\end{equation}
An alternative way to bound part (b) is to note that $V$ are the eigenvectors of $X^T X$ and $\Vhat$ are the eigenvectors of $Z^TZ = X^TX + \Delta$, where $\Delta = X^TE+E^TX+E^TE$. Then the difference 
\[
\|\Delta\| \leq \|X^TE+E^TX+E^TE\| \leq 2\|X^TE\| + \|E^TE\|. 
\]
We bound the right-hand side. 
Note that $X^T E = V\Lambda U^T E$, where each column of $U^T E/\sigma \in \reals^{r \times d}$ is an independent, centered, and isotropic sub-Gaussian vector. By Theorem 5.39 of  \cite{vershynin}, with probability at least $1 - O(1/n^2)$, there are updated $C_1(K_{\psi_2})$ and $C_2(K_{\psi_2},r)$, so that 
$\|U^TE\| \leq \sigma(\sqrt{d} + C_1\sqrt{r} + \sqrt{\log n})$, and 
\[
\|X^T E\| \leq \|\Lambda\|\|U^TE\| \leq \frac13\sigma\sqrt{n} C_2(\sqrt{d}+\sqrt{\log n}).
\]
Again, the same theorem can be applied to $E$, which shows that with probability at least $1 - O(1/n^2)$, 
\[
\|E^TE\| \leq C_1^2\sigma^2(\sqrt{n}+\sqrt{d}+\sqrt{2\log n})^2
\leq 4C_1^2\sigma^2(\sqrt{n}+\sqrt{d})^2.
\]
Combine them. It follows that we can update constant $C_2(K_{\psi_2},r)$ so that 
\begin{align*}
\|\Delta\| &\leq \tfrac{2}{3}C_2(\sigma\sqrt{n d}+\sigma\sqrt{n\log n})+4C_1^2\sigma^2(n+2\sqrt{nd}+d)\\
&\leq \tfrac23C\sigma (\sqrt{nd}+\sqrt{n\log n}+\sigma n+\sigma \sqrt{nd}+\sigma d)\\
&\leq C_2\sigma (\sqrt{nd}+\sqrt{n\log n}+\sigma n).
\end{align*}
The last step is achieved because we assume $\sigma \sqrt{d}\leq \lambda_r\leq \|X\|=\sqrt{n}$. 
Therefore, by the Davis--Kahan Theorem \citep{daviskahan}, there exists an orthogonal matrix $O$ and updated $C_2(K_{\psi_2},r)$, so that 
\begin{equation}\label{eqn:vsintheta2}
    \|\Vhat_r - VO\|\leq
    \frac{\|\Delta\|}{\lambda_r^2}
    \leq \frac{C_2\sigma}{\lambda_r^2} (\sqrt{nd}+\sqrt{n\log n}+\sigma n).
\end{equation}
Since \eqref{eqn:vsintheta} and \eqref{eqn:vsintheta2}  both hold, it suffices to use the minimum of them as the upper bound of part (b). 

\textbf{Part (a).} 
$\|\xi_i^T\Vhat_r\|$ is the most challenging part because of the correlation between $\xi_i$ and $\Vhat_r$. If we ignore the correlation and bound it by $\|\xi_i\|\|\Vhat_r-VO\|$, then the bound is of order $\sigma^2\sqrt{d}$. For high-dimensional data, this $d$ can be very large. Therefore, we need a delicate analysis of this term.

First, since the rows of $E$ are independent sub-Gaussian distributed, 
by Theorem 5.39 in \cite{vershynin}, there is a constant $C_1=C_1(K_{\psi_2})$, so that with probability at least $1 - O(1/n^3)$, 
\[
|X_i^T\xi_i|\leq \sigma (1+\tfrac16C_1(1+\sqrt{\log n}))
\leq \frac13C_1\sigma (1+\sqrt{\log n}),
\]
\[
\|\xi_i\|\leq \sigma (\sqrt{d}+\tfrac16C_1(1+\sqrt{\log n}))
\leq \frac13C_1\sigma (\sqrt{d}+\sqrt{\log n}),
\]
\[
\|E\| \leq \sigma(\sqrt{d} + \tfrac16C_1(\sqrt{n} + \sqrt{\log n}))
\leq \frac{1}{3}C_1\sigma(\sqrt{d} +\sqrt{n}).
\]
We will define the event that all these hold as $\mathcal{A}$. The following discussions are conducted when $\mathcal{A}$ happens. 
Let $\lambdahat_k = \lambda_k(Z)$ be the $k$-th largest singular value of $Z = X+E$ and $\vhat_k$ be the corresponding right singular vector in $\Vhat$. By Weyl's inequality, the singular values follow that, for $k \leq r$, 
\begin{equation}\label{eqn:looeigen2}
\sqrt{n}+\frac13C_1\sigma(\sqrt{n}+\sqrt{d})\geq \lambdahat_k\geq \lambda_r - \frac13C_1\sigma(\sqrt{n}+\sqrt{d}), \quad \lambdahat_{r+1} \leq \frac13C_1\sigma(\sqrt{n}+\sqrt{d}).    
\end{equation}

Consider the $k$-th entry of $\xi_i^T\Vhat_r \in \reals^r$. 
The $k$-th entry of $\xi_i^T \Vhat_r$ is $\xi_i^T \vhat_k$. By definition of singular vectors, we have 
\[
(X+E)^T(X+E)\vhat_k = \lambdahat_k^2 \vhat_k.
\]
Multiply $\xi_i^T$ to both sides, and we have 
\begin{align*}
\lambdahat_k^2\xi_i^T \vhat_k&=\xi_i^T (X+E)^T(X+E)\vhat_k\\
&=\xi_i^T\bigl[(X_i+\xi_i)(X_i+\xi_i)^T + \sum_{j\neq i}(X_j+\xi_j)(X_j+\xi_j)^T\bigr]\vhat_k\\
&=\xi_i^TX_iX_i^T \vhat_k+
\xi_i^TX_i\xi_i^T \vhat_k+\xi_i^T\xi_iX_i^T \vhat_k +\|\xi_i\|^2\xi_i^T\vhat_k + \xi_i^TS_i \vhat_k,
\end{align*}
where $S_i=\sum_{j\neq i}(X_j+\xi_j)(X_j+\xi_j)^T$.
We rewrite this as 
\begin{equation}\label{eqn:xiivj}
(\lambdahat_k^2-\|\xi_i\|^2)\xi_i^T \vhat_k=\xi_i^TS_i \vhat_k+\xi_i^TX_iX_i^T \vhat_k+
\xi_i^TX_i\xi_i^T \vhat_k+\xi_i^T\xi_iX_i^T \vhat_k. 
\end{equation}
Now we analyze the terms on the right-hand side of \eqref{eqn:xiivj}. 

The first term $\xi_i^TS_i \vhat_k$ is the most challenging part since $S_i$ is dependent on $\vhat_k$. Note that $S_i$ is the version of $(X+E)^T(X+E)$ when sample $i$ is deleted. 
Denote the eigenvalue decomposition of $S_i = \Vtilde\tilde{\Lambda}^2\Vtilde^T$. Let $\Vtilde_r$ be the first $r$ columns of $\Vtilde$. Proposition \ref{prop:leave1} in below states that if $y_k=\Vtilde_r^T\vhat_k\in \reals^r$, with probability $1-O(1/n^2)$
\begin{equation}\label{eqn:propvk}
\|\vhat_k-\Vtilde_ry_k\|=\|(I-\Vtilde_r \Vtilde_r^T)\vhat_k\|\leq \frac{C_2}{\lambda_r^2}(1+\sigma\sqrt{\log n})(1+\sigma(\sqrt{d}+\sqrt{\log n})). 
\end{equation}
We decompose $\vhat_k$ into $\Vtilde_r y_k$ and the remainder, and then it follows  
\[
\xi_i^TS_i \vhat_k=
\xi_i^TS_i \Vtilde_ry_k+
\xi_i^TS_i(\vhat_k-\Vtilde_ry_k)
=\xi_i^T\Vtilde_r\Lambdatilde_r^2y_k+
\xi_i^TS_i(\vhat_k-\Vtilde_ry_k).
\]
Since $\Vtilde$ is based on $S_i$ where sample $i$ is deleted, so $\xi_i$ is independent with $\Vtilde_r$. Further, recall $\xi_i$ is  sub-Gaussian,
using Hanson--Wright inequality (Theorem 1.1 in \cite{rudelson2013hanson}), there is an absolute constant $c>0$ so that 
\[
P(|\|\Vtilde_r^T\xi_i\|^2-\sigma^2r|>\sigma^2 t)\leq 2\exp\left(-c\min(\frac{t^2}{K_{\psi_2}^4r}, \frac{t}{K_{\psi_2}^2})\right),
\]
We pick $t=\sqrt{r}K_{\psi_2}^2\log n/\min\{c,\sqrt{c}\}$ then the probability will be less than $O(1/n^2)$ for sufficiently large $n$. In other words, we can update $C_2=C_2(K_{\psi_2},r)$
\begin{equation}\label{eqn:xivo}
\|\xi_i^T\Vtilde_r\|\leq \sigma C_2(1+\sqrt{\log n}), 
 \end{equation}
with $1-O(1/n^2)$ probability. We assume \eqref{eqn:xivo} takes place in the following discussion. 
% then there is $\Vtilde_r^T \xi_i\sim \mathcal{N}(0, \sigma^2 I_r)$. By the Laurent-Massart bound in \cite{quad}, with probability at least $1 - 1/n^2$,
Then consider $\Lambdatilde_r^2$, which contains the eigenvalues of $S_i$. Because $S_i=(X+E)^T(I-e_i e_i^T)(X+E)$, its eigenvalues are dominated by the ones of $(X+E)^T(X+E)$. So  with probability $1 - O(1/n^2)$,  the eigenvalues of $\Lambdatilde_r^2$ satisfy the following with some updated $C_1(K_{\psi_2})$ using \eqref{eqn:looeigen2},
\begin{equation}\label{eqn:tildelambda}
\lambdatilde_k \leq \left\{\begin{array}{ll}
\lambdahat_1^2\leq C_1 n & k\leq r;\\
 \lambdahat_{r+1}^2\leq C_1(\sqrt{n}+\sqrt{d})^2\sigma^2,& k>r.
\end{array}
\right.
\end{equation}
Combine \eqref{eqn:xivo}, \eqref{eqn:tildelambda} and $\|y_k\| \leq 1$,  we can update $C_1(K_{\psi_2})$ so that
\begin{equation}\label{eqn:term1}
  \|\xi_i^T\Vtilde_r\Lambdatilde^2_ry_k\|
\leq \|\xi_i^T\Vtilde_r\|\|S_i\|\leq C_1\sigma n\sqrt{\log n}.  
\end{equation}
Now consider the term $\xi_i^T S_i(\vhat_kk - \Vtilde_r y_k)$. 
Since $S_i$ and $\xi_i$ are independent, using Hanson--Wright inequality (Theorem 1.1 in \cite{rudelson2013hanson}), there is an updated $C_1(K_{\psi_2})$ so that with probability $1 - O(1/n^2)$, 
\begin{eqnarray*}
\|\xi_i^TS_i\|^2 & \leq & \sigma^2 \tr(S_i^2) + C_1 \sigma^2(\|S_i^2\|_F + \|S_i^2\|)\log n \\
& \leq & C_1\sigma^2(\tr(S_i^2) + \|S_i^2\|_F)\log n.
\end{eqnarray*}
The term $\|S_i^2\| = \lambda_1(S_i^2) \leq \tr(S_i^2)$, so it is ignored. Now recall that $\tr(S_i^2) = \sum_{i=1}^d \lambdatilde_i^2$ and $\|S_i^2\|_F = (\tr(S_i^4))^{1/2} = (\sum_{i=1}^d \lambdatilde_i^4)^{1/2}$. Combining it with the results about $\lambdatilde_i$ in \eqref{eqn:tildelambda}, we can update $C_2$ so that 
\begin{eqnarray}\label{eqn:xisi}
\|\xi_i^TS_i\|
&\leq&  C_1\sigma \sqrt{\sum\nolimits_{k=1}^d \lambdatilde_k^2 + \sqrt{\sum\nolimits_{k=1}^d \lambdatilde_k^4}}\sqrt{\log n}\nonumber\\
&\leq& 2C_1\sigma \sqrt{\sum\nolimits_{k=1}^d \lambdatilde_k^2}\sqrt{\log n}\nonumber\\
&\leq& \frac12C_2\sigma \sqrt{n^2+(\sqrt{n}+\sqrt{d})^4\sigma^4n} \sqrt{\log n}\nonumber\\
& \leq & C_2\sigma \sqrt{\log n}(n+\sigma^2 n\sqrt{n}+\sigma^2 d\sqrt{n}).
\end{eqnarray}
Combining it with \eqref{eqn:propvk} from Proposition \ref{prop:leave1}, we have an updated $C_2(K_{\psi_2},r)$, so that with probability $1 - O(1/n^2)$,  
\[
\|\xi_i^TS_i(\vhat_k-\Vtilde_ry_k)\|
\leq \Gamma(1+\sigma(\sqrt{d}+\sqrt{\log n})),\quad\Gamma:=\frac{C\sigma}{\lambda_r^2}\sqrt{\log n}(n+\sigma^2n\sqrt{n}+\sigma^2 d\sqrt{n})(1+\sigma\sqrt{\log n}).
\]

Consider the other three terms in \eqref{eqn:xiivj}. Using the definition of event $\mathcal{A}$, 
with probability $1-O(1/n^2)$ and updated $C_1(K_{\psi_2})$, we have 
\begin{align*}
&\|\xi_i^TX_iX_i^T\|\leq \|X_i^T\xi_i\| \leq \sigma C_1\sqrt{\log n},\\
&\|\xi_i^TX_i\xi^T_i\| \leq  \xi_i^T\xi_i \|X_i\| \leq \xi^T_i\xi_i
\leq C_1\sigma^2(\sqrt{d}+\sqrt{\log n})^2,\\
&\|\xi_i^T\xi_iX^T_i\|\leq \xi_i^T\xi_i\|X_i\| \leq \xi_i^T\xi_i \leq C_1\sigma^2(\sqrt{d}+\sqrt{\log n})^2.
\end{align*}
Combine all the four terms together in \eqref{eqn:xiivj}, where $\xi_i^TS_i\vhat_k$ has been split into \eqref{eqn:term1} and $\Gamma(1+\sigma(\sqrt{d}+\sqrt{\log n}))$.  Now we have
\begin{eqnarray*}
(\lambdahat_k^2-\|\xi_i\|^2)|\xi_i^T \vhat_k|
& \leq & \Gamma(1+\sigma(\sqrt{d}+\sqrt{\log n})) + C_1d\sigma^2+C_1 n\sigma \sqrt{\log n}.
\end{eqnarray*}
Also recall that when $\mathcal{A}$ takes place, $\lambdahat_k^2-\|\xi_i\|^2\geq \lambda^2_r/4$. Therefore, we can update $C_1(K_{\psi_2})$ so that 
\begin{equation}\label{eqn:maina}
|\xi_i^T \vhat_k|\leq \frac{4}{\lambda_r^2}\Gamma(1+\sigma(\sqrt{d}+\sqrt{\log n}))+\frac{C_1 d\sigma^2+ nC_1\sigma \sqrt{\log n}}{\lambda_r^2}.
\end{equation}

{
\textbf{Conclusion.}
The first error bound can be reached by removing the lower-order terms. Recall that  $\sqrt{n}\geq \lambda_r\geq 1+C_1\sigma(\sqrt{n}+\sqrt{d}),$
\[
\frac{d\sigma^2}{\lambda_r^2}\leq 
\frac{\sigma \sqrt{nd}}{C_1\lambda^2_r},\quad \frac{d\sigma^2}{\lambda_r^2}\leq 
\frac{\sigma\sqrt{d}}{C_1\lambda_r},\quad 
\frac{\sigma}{\lambda_r^2} \sqrt{n\log n}\leq \frac{\sigma}{\lambda_r^2} n\sqrt{\log n},\quad \frac{\sigma}{\lambda_r^2}\sigma n\leq \frac{\sigma}{\lambda_r^2} n\sqrt{\log n}.
\]
Next, if we fix $\lambda_r \geq c_X \sqrt{n}$ and $d\gtrsim n$, then our previous upper bound can be simplified because 
\begin{align*}
&
\frac{\sigma(\sqrt{n}+\sqrt{d})}{\lambda_r}+\frac{\sigma n\sqrt{\log n}}{\lambda_r^2} +\frac{1}{\lambda_r^2}\Gamma(1+\sigma(\sqrt{d}+\sqrt{\log n}))\\
&\lesssim \sqrt{\frac{d}{n}}\sigma(1 +\frac{\sigma^3 d}{n}\sqrt{\log n} +\frac{\sigma^4 d}{n}\log n)+\sqrt{\log n}\sigma.
\end{align*}
Similarly, if we fix $\lambda_r \geq c_X \sqrt{n}$ and $d\lesssim n$, then our previous upper bound can be simplified because
\begin{align*}
\frac{\sigma}{\lambda_r^2} (\sqrt{nd}+n\sqrt{\log n}) +\frac{1}{\lambda_r^2}\Gamma(1+\sigma(\sqrt{d}+\sqrt{\log n}))\lesssim
\sqrt{\log n}\sigma(1+\sigma^4 \sqrt{\frac{d\log n}{n}}).
\end{align*}
% \sqrt{\frac{d}{n}}\sigma+
% \sqrt{\frac{\log n}{n}}\sigma+\sigma^2 +\frac{\sigma^2}{n^2}\sqrt{\log n}n\sqrt{d}+\frac{\sigma^2}{n^2}\sqrt{\log n}\sigma^2 n\sqrt{nd}+
% \frac{\sigma^3}{n^2}\log n n\sqrt{d}+\frac{\sigma^3}{n^2}\log n \sigma^2 n\sqrt{nd}\\
% &\lesssim
% \sqrt{\frac{d}{n}}\sigma+
% \sqrt{\frac{\log n}{n}}\sigma+\sigma^2+
% \sqrt{\frac{d}{n}}\sigma(\frac{\sigma}{\sqrt{n}}\sqrt{\log n}+\sigma^3\sqrt{\log n}+
% \frac{\sigma^2}{\sqrt{n}}\log n +\sigma^4\log n )
The detailed derivation of these simplifications can be found in Lemma 1 of \cite{supp}. 
These conclude the second claim. For the third claim, simply note when $\sigma\leq 1$,
\[
\sqrt{\log n}\sigma(1+\sigma^4 \sqrt{n^{-1}{d\log n}})\leq 2\log n\sigma.
\]
Finally, our claims work for every $i \in [n]$. Take the maximum over all $i\in [n]$, and with probability $1 - O(1/n)$, the bound holds. }
\end{proof}

\subsection{Leave-one-out eigenvector perturbation}\label{sec:loo}
In the proof of the main theorem, we use the leave-one-out method to deal with the most challenging term.  
This method considers leaving out one row or column of a random matrix and investigates how would the matrix change. It has been used to study the random matrix spectrum through Stieltjes transform \citep{tao2023topics} and 
PCA eigenvector behavior \citep{chen2021spectral,o2023optimal}. 
Let $\vhat_k$ be the $k$-th right singular vector of the data matrix without one specific sample and $\Vtilde_r$ be the right singular matrix of the original matrix. We show that  $\vhat_k$ and its projection in the column space of $\Vtilde_r$ differs by ${O}(\sigma \sqrt{d \log n}/\lambda_r^2)$, when \eqref{eqn:specgap} holds. It delicately describes the contribution of one sample to the singular vectors.
\cite{chen2021spectral} studied the same problem and established a bound in (4.19). The bound requires $X$ and $E$ to be symmetric, and $\|\vhat_k\|_\infty=O({1}/{\sqrt{d}})$ to be effective. Our analysis does not rely on $\|\vhat_k\|_\infty$ or the symmetry, which is more general and allows arbitrary $X$. 
% It is to set up the upper bound of the difference between the singular vectors of the original data matrix and the matrix without one specific sample.
Such results can be useful in other applications. We rigorously present the result as the following proposition.

\begin{prop}
\label{prop:leave1}
 Suppose $X\in \reals^{n\times d}$ has rank $r$, where each row  $X^T_i$ is bounded by $\|X_i\|\leq 1$. Suppose $E\in \reals^{n\times d}$ has all entries being independent  sub-Gaussian$(0,\sigma^2,K_{\psi_2})$ distributed. 
Let $\Ztilde$ be the last $n-1$ rows of $Z=X+E$. Denote the SVD of $X=U \Lambda V^T$, $Z=\Uhat \Lambdahat \Vhat^T$ and $\Ztilde=\tilde{U}\tilde{\Lambda}\Vtilde^T$. Let $\Vhat_r$ and $\Vtilde_r$ denote the first $r$ columns of $\Vhat$ and $\Vtilde$, respectively. Let $\lambda_r$ be the smallest nonzero singular value of $X$. 
Then there are constants $n_0, C_1=C_1(K_{\psi_2})$ and $C_2=C_2(r,K_{\psi_2})$ so that if
\[n>n_0,\quad 
\lambda_r> C_1+C_1\sigma(\sqrt{n}+\sqrt{d}),
\]
the following holds with probability at least $1 - O(1/n^2)$,
\[
\max_{k\leq r}\|(I-\Vtilde_r \Vtilde_r^T)\vhat_k\|\leq \begin{cases}
    C_1{\lambda^{-2}_r},& (\sqrt{d}+\sqrt{\log n})\sigma \leq 2;\\
   C_2{\lambda_r^{-2}}\sigma(1+\sigma\sqrt{\log n})(\sqrt{d}+\sqrt{\log n}),& (\sqrt{d}+\sqrt{\log n})\sigma>2,
\end{cases}
\]
and the following holds for both cases:
\[
\max_{k\leq r}\|(I-\Vtilde_r \Vtilde_r^T)\vhat_k\|\leq {C_2}{\lambda_r^{-2}}(1+\sigma\sqrt{\log n})(1+\sigma(\sqrt{d}+\sqrt{\log n})). 
\]
% If further we  have $\lambda_r \geq c_X\sqrt{n}$, then there exists $\sigma_0 > 0$, so that for $\sigma \leq \sigma_0$, $\|\blue{(I-\Vtilde_r \Vtilde_r^T)\vhat_k}\| = O(\sigma\sqrt{d}/n)$ and $\|\blue{(I-\Vtilde_r \Vtilde_r^T)\vhat_k}\| = O(1/n)$ when $\sqrt{d}\sigma = O(1)$.
\end{prop}

To find out the difference between the original eigenvector and its projection in eigen-space formed by the leading $r$ right singular vectors of the leave-one-out sample, we take advantage of the randomness of the perturbation. Roughly speaking, we project the eigenvector equation for $\hat{v}_k$ to the eigen-space of interest and its orthogonal complement, where the latter is the difference of interest. By controlling the random effects of one sample on the eigen-space, we show that the difference has a small norm.
%\blue{Roughly speaking, the main idea of the proof is to project the eigenvector equation of $\vhat_k$ onto $\Vtilde_r$ and its orthogonal complement, and take advantage of the randomness of the perturbation to show the latter projection is small. This idea can also be found in other contexts in the literature, e.g. in the proof of Lemma 9 of \cite{o2023optimal}. }

\begin{proof} 
We discuss the problem for both the small noise case and the large noise case. We fix $k\leq r$ in the discussion. The claim for all $k$ holds by a union bound. 

\textbf{Case 1: the small $\sigma$.}
Since the rows of $E$ are independent sub-Gaussian distributed, 
by Theorem 5.39 in \cite{vershynin}, there is a constant $C_1=C_1(K_{\psi_2})$, so that with probability at least $1 - O(1/n^3)$, 
\begin{eqnarray*}
   \|\xi_1\|\leq \sigma (\sqrt{d}+\tfrac16C_1(1+\sqrt{\log n}))
\leq \frac13C_1\sigma (\sqrt{d}+\sqrt{\log n}),\\ 
\|E\| \leq \sigma(\sqrt{d} + \tfrac16C_1(\sqrt{n} + \sqrt{\log n}))
\leq \frac{1}{3}C_1\sigma(\sqrt{d} +\sqrt{n}).
\end{eqnarray*}
We will define the event that both these hold as $\mathcal{A}$. Our discussion assumes $\mathcal{A}$ holds. 

Denote
\[
\Delta=(X+E)^T(X+E)-(\Xtilde+\Etilde)^T(\Xtilde+\Etilde)
=(X_1+\xi_1)(X_1+\xi_1)^T. 
\]
Since $\|X_1\| \leq 1$, 
\[
\|\Delta \|_F=\|X_1+\xi_1\|^2\leq 2+C_1^2\sigma^2 (d+\log n). 
\]
So by Weyl's inequality, the singular values follow that, for $k \leq r$, 
\begin{equation}\label{eqn:looeigen}
\lambdahat_k\geq \lambda_r - \frac13C_1\sigma(\sqrt{n}+\sqrt{d}), \quad \lambdahat_{r+1} \leq \frac13C_1\sigma(\sqrt{n}+\sqrt{d}).    
\end{equation}
Hence, by the Davis--Kahan Theorem \citep{daviskahan}, there is an orthogonal matrix $O$ so that 
\[
\|\Vhat_r-\Vtilde_rO\|_F\leq \frac{4\|\Delta\|_F}{\lambdahat^2_r-\lambdahat_{r+1}^2}\leq \frac{8+4C_1^2\sigma^2(d+\log n)}{\lambda_r(\lambda_r -\tfrac23C_1\sigma(\sqrt{n}+\sqrt{d}))}\leq \frac{24(1+C_1^2\sigma^2(d+\log n))}{\lambda_r^2}.
\]
Let $y_k$ be the $k$-th column of $O$, then $\vhat_k-\Vtilde_r y_k$ is the $k$-th column of $\Vhat_r-\Vtilde_rO$. For any column, the $\ell_2$ norm is bounded by the matrix Frobenius norm. Therefore, the result follows by 
\[
\|(I-\Vtilde_r\Vtilde^T_r)\vhat_k\|
=\min_{w_k\in \reals^r}\|\vhat_k -\Vtilde_r w_k\|\leq \|\vhat_k -\Vtilde_r y_k\|\leq \|\Vhat_r-\Vtilde_rO\|_F,\quad \sigma(\sqrt{d}+\sqrt{\log n})\leq 2,  
\]
and updating the value of $C_1$.

\textbf{Case 2: the large $\sigma$.} 
This is a challenging case. We first introduce the notations. Recall the observed data is $Z = X+E$. We compare it with the case that sample $1$ is missing, so we define $\Ztilde = \Xtilde+\Etilde$ and the first sample is $Z_1 = X_1 + \xi_1$. 
The eigen-decomposition is $\Ztilde^T \Ztilde = \Vtilde \Lambdatilde^2\Vtilde^T$, where $\Vtilde = [\Vtilde_r, \Vtilde_\bot]$.

We define $y_k$ and $y_{k, \bot}$ as follows:
\[
y_k=\Vtilde_r^T\vhat_k,\quad 
y_{k,\bot}=\Vtilde_\bot^T\vhat_k.
\]
Since $\|(I-\Vtilde_r\Vtilde^T_r)\vhat_k\|=\|\Vtilde_\bot\Vtilde^T_\bot\vhat_k\|\leq \|y_{k,\bot}\|$, we are interested in bounding $\|y_{k,\bot}\|$. To do it, note that $\vhat_k$ is the eigenvector of $Z^TZ$ corresponding to the eigenvalue $\lambdahat_k$, so 
 $Z^TZ\vhat_k=\lambdahat_k^2\vhat_k$. Decompose $Z$ into $\Ztilde$ and $Z_1$, then it becomes 
\[
(\Ztilde^T\Ztilde+Z_1Z_1^T)(\Vtilde_r y_k+\Vtilde_\bot y_{k,\bot})=\lambdahat_k^2(\Vtilde_r y_k+\Vtilde_\bot y_{k,\bot}).
\]
Multiplying by $\Vtilde^T = [\Vtilde_r,\Vtilde_\bot]^T$ on both sides leads to:
\begin{equation}
\label{tmp:2block}
\begin{bmatrix}
\Lambdatilde_r^2+\Vtilde_r^TZ_1Z_1^T\Vtilde_r & \Vtilde_r^TZ_1Z_1^T\Vtilde_\bot \\
\Vtilde_\bot^TZ_1Z_1^T\Vtilde_r & \Lambdatilde_\bot^2+\Vtilde_\bot^TZ_1Z_1^T\Vtilde_\bot 
\end{bmatrix}\begin{bmatrix}
y_k\\
y_{k,\bot}
\end{bmatrix}=
\begin{bmatrix}
\lambdahat_k^2 y_k\\
\lambdahat_k^2 y_{k,\bot}
\end{bmatrix},
\end{equation}
where $\Lambdatilde_r$ is the diagonal matrix consisting of the first $r$ singular values of $\Ztilde$,  and  $\Lambdatilde_\bot$ is the diagonal matrix consisting of the remaining singular values of $\Ztilde$.

Denote $a=\Vtilde_r^TZ_1$ and $b=\Vtilde_\bot^TZ_1$. From the second row of \eqref{tmp:2block}, we have 
\[
y_{k,\bot}=(\lambdahat^2_k I - (\Lambdatilde^2_\bot+bb^T ))^{-1}ba^Ty_k.
\]
So $\|y_{k,\bot}\|$ can be bounded by analyzing the terms $a, b$, $\lambdahat_k$ and $\Lambdatilde^2_\bot$. 

Note that 
\[
\E_{\xi_1}[\|\Vtilde_r^T\xi_1\|^2]=\sigma^2 \text{tr}(\Vtilde_r^T\Vtilde_r)=\sigma^2 \|\Vtilde_r\|^2_F=\sigma^2r,\quad \|\Vtilde_r\|=1.
\]
By the independence between $\Vtilde$ and $Z_1 = X_1 + \xi_1$, using Hanson--Wright inequality (Theorem 1.1 in \cite{rudelson2013hanson}), we find that 
\[
P(|\|\Vtilde_r^T\xi_1\|^2-\sigma^2r|>\sigma^2 t)\leq 2\exp\left(-c\min(\frac{t^2}{K_{\psi_2}^4r}, \frac{t}{K_{\psi_2}^2})\right).
\]
We pick $t=3K_{\psi_2}^2\log n/c$, then the probability is $O(1/n^3)$ for sufficiently large $n$. Likewise,
\[
P(|\|\Vtilde_\bot^T\xi_1\|^2-\sigma^2(d-r)|>\sigma^2 t)\leq 2\exp\left(-c\min(\frac{t^2}{K_{\psi_2}^4(d-r)}, \frac{t}{K_{\psi_2}^2})\right).
\]
If we pick $t=3K_{\psi_2}^2\sqrt{d\log n}/\sqrt{c}$, the probability above will be $O(1/n)$. In conclusion, there is a constant $C_1=C_1(K_{\psi_2}), C_2=C_2(K_{\psi_2},r)$, so that the following holds with probability at least $1 - O(1/n^3)$,  
\begin{align*}
&\|a\|\leq \|X_1\|+\|\Vtilde_r^T\xi_1\|\leq 1+\tfrac12C_1\sigma(\sqrt{r}+\sqrt{\log n})\leq C_2(1+\sigma \sqrt{\log n}),  \\
&\|b\|\leq \|X_1\|+\|\Vtilde_\bot^T\xi_1\|\leq \frac{1}{6} + \frac16C_1\sigma(\sqrt{d}+\sqrt{\log n})\leq \frac{1}{6}\lambda_r. 
\end{align*}
Assuming $\mathcal{A}$ holds, we can combine these estimates with \eqref{eqn:looeigen} on $\lambdahat_k$ and $\lambdahat_{r+1}$, so that 
\begin{align*}
\|(\lambdahat_k^2 I-\Lambdatilde^2_\bot+bb^T)^{-1}b\|&\leq \frac{\|b\|}{\lambdahat_r^2 - \lambdahat_{r+1}^2 - \|b\|^2}\leq \frac{1 + C_1\sigma(\sqrt{d}+\sqrt{\log n})}{2\lambda_r^2 -\lambda_r^2}.
\end{align*}
Therefore, we have
\begin{eqnarray*}
\|y_{k,\bot}\| & \leq & 
    \|(\lambdahat_k^2 I-\Lambdatilde^2_\bot+bb^T)^{-1}b\|\|a\|\|y_k\| \leq \frac{C_2}{\lambda_r^2}(1+\sigma\sqrt{\log n})(1+C_1\sigma (\sqrt{d}+\sqrt{\log n})). 
\end{eqnarray*}
When $\sigma \sqrt{d}\geq 2$, the latter factor can be further simplified. The final claim can be achieved with an updated value of $C_2$. 

\end{proof}

\subsection{Condition on the minimum eigenvalue and zigzag lines}\label{sec:eigenvalue}
The main theorem in Section \ref{sec:unif} provides an upper bound on $\err{X - \Xhat}$
%in relation to $\lambda_r$ and $\sigma$. 
when the smallest nonzero singular value of the clean data matrix $X$ is at the order of $\sqrt{n}$. Hence, it becomes pertinent to ask under what conditions this criterion can be satisfied.

We introduce a theorem below concerning low-rank random matrices. Let $X$ be a random matrix with $cov(X_i) = \Sigma$, where $\Sigma$ is a low-rank matrix. The smallest nonzero singular value of $X$ can be lower bounded by the smallest eigenvalue of $\Sigma$ and $\sqrt{n}$. This result can be viewed as a straightforward extension of the standard random matrix theory as \cite[Theorem 5.39]{vershynin}.

\begin{thm}\label{thm:eigenvalue}
Suppose $X_i\in\reals^d$ are i.i.d. samples from a distribution so that $\|X_i\|\leq 1$ almost surely,  with $\E[X_i] = 0$ and $\cov(X_i)=\Sigma$ where $rank(\Sigma) = r$. Let $\lambda_i(\Sigma)$ be the $i$-th largest eigenvalue of $\Sigma$. If $n>r$, the following holds with  probability at least $1 - e^{-ct^2}$, 
\[
\sqrt{\lambda_r(\Sigma)}(\sqrt{n}-C\sqrt{r}-t)\leq \lambda_r(X)\leq \lambda_1(X)\leq \sqrt{n}.
\]
Here $C$ and $c$ are some constants that depend on $\lambda_r(\Sigma)$, but not on $n$ and $d$.
\end{thm}
\begin{proof}
Since $\|X_i\| \leq 1$ almost surely, the Frobenius norm of $X$ is bounded, and hence the spectral norm $\|X\| = \lambda_1(X)$ can be bounded as follows: 
\[
\lambda^2_1(X)\leq \|X\|_F^2=\sum_{i=1}^n \|X_i\|^2\leq n. 
\]
Consider the lower bound of $\lambda_r(X)$. Denote the eigen-decomposition of $\Sigma$ be $\Sigma= U_{\Sigma}\Lambda_{\Sigma} U_{\Sigma}^T$ where $U_{\Sigma}\in \reals^{d\times r}$ and $\Lambda_{\Sigma}\in \reals^{r\times r}$. 
We set $Y_i=\Lambda_{\Sigma}^{-1/2}U_{\Sigma}^TX_i\in \reals^r$, then $\|Y_i\|\leq [\lambda_r(\Sigma)]^{-1/2}$. Further, $\E[Y_i] = 0$, $cov(Y_i) = I_r$, and $Y_i$ is sub-Gaussian distributed, where the constant $K_{\psi_2}$ follows 
\[
K_{\psi_2}(Y_i)=\sup_{\|y\|=1}\sup_{p\geq 1}\frac{\E[|\langle y,Y_i\rangle|^p]^{1/p}}{p^{1/2}}
\leq 
\sup_{\|y\|=1}\sup_{p\geq 1}\frac{[\lambda_r(\Sigma)]^{-1/2}}{p^{1/2}} = [\lambda_r(\Sigma)]^{-1/2}. 
\]
We can apply \cite[Theorem 5.39]{vershynin} to $Y$ and find constant $C$ and $c$ so that with probability $1-2\exp(-ct^2)$, 
\[
\sqrt{n}-C\sqrt{r}-t\leq \lambda_r(Y).
\]
Finally, according to the definition of $Y$, the original matrix $X$ of interest is $X=Y\Lambda_{\Sigma}^{1/2}U_{\Sigma}^T$.  
Therefore, the smallest singular value of $X$ can be bounded by 
\[
\lambda_r(X)\geq \lambda_r(Y)\sqrt{\lambda_r(\Sigma)} \geq \sqrt{\lambda_r(\Sigma)}(\sqrt{n}-C\sqrt{r}-t). 
\]
The result is proved.
\end{proof}

By Theorem \ref{thm:eigenvalue}, it suffices to compute the covariance of the clean data $X_i$ and show its $r$-th eigenvalue is bounded from below. 
This can be done easily if $X_i$ follows some known distributions such as Gaussian mixtures. Below, we discuss a scenario where the covariance computation is not so elementary. 

%\subsubsection{Zigzag line}\label{sec:zigzag}
For many temporal and spatial data sets \citep{cressie2015statistics,brunton2022data, khoo2024temporal}, researchers expect the clean data to be generated from a one-dimensional curve, i.e. the function of time or location $t$. It is of interest to estimate the function $x_t$, where a popular approach is to use a piecewise linear approximation \citep{pavlidis1974segmentation, stone1961approximation}.
For the simplicity of exposition, we assume $x_t$ is piecewise linear.
%For many temporal and dynamics related problems \cite{cressie2015statistics,brunton2022data}, we would expect the clean data is generated from a one dimensional curve $x_t$, $t\in [0,1]$. \blue{For the simplicity of exposition,} we assume $x_t$ is piecewise linear. 
In less mathematical terms, $x_t$ is a zigzag line. Figure \ref{fig: zig zag 3D} below shows two such objects in 3-dimensional space. For each zigzag line, there will be $R$ time points $0=t_1<t_2<t_3<\ldots<t_R<1=t_{R+1}$, and
\[
x_s=x_{t_j}+(s-t_j)v_j,\quad s\in [t_j,t_{j+1}].
\]
Here $\{v_j,j=1\ldots,R\}$ is a group of $d$-dimensional vectors with unit-norm. We want to make a special note here that the number of segments $R$ does not indicate the dimension of the sub-space dimension. For example, Figure \ref{fig: zig zag 3D}  shows two zigzag lines in a three-dimensional subspace ($r = 3$) with any $R=10$ segments. 

\begin{prop}\label{prop:zigzag}
Suppose $X_i=x_{s_i}$ and $s_i\sim Unif[0,1]$ independently, $i \in [n]$.
Suppose the linear subspace spanned by $\{v_j,j=1\ldots,R\}$ has dimension $r$ and $\E[X_i]=0$. If there is a constant $\rho > 0$ so that the segment endpoints $t_{i+1}-t_i\geq \rho$, then there exists a constant $c > 0$ independent of $n$ and $d$, so that the following holds with probability $1 - O(1/n)$, 
    \[
    c\sqrt{\rho^3/3 - \rho^4/4}\sqrt{n} \leq \lambda_r(X)\leq\lambda_1(X)\leq \sqrt{n}.
    \]
\end{prop}
Proposition \ref{prop:zigzag} utilizes Theorem \ref{thm:eigenvalue} to establish that the minimum singular value $\lambda_r(X) > c_X\sqrt{n}$ holds for the zigzag example, given that each segment has a minimum length $\rho > 0$ and no segment is degenerate.

In the case where there are degenerate segments, denoted by $\rho = o(1)$, we can still provide an upper bound on the error. Notably, Theorem \ref{thm:linftydavis} demonstrates that the error of Algorithm \ref{alg:denoise} can be controlled by $n\sigma/\lambda_r^2(X)$ for any arbitrary $\lambda_r(X)$. By incorporating Proposition \ref{prop:zigzag} into this result, we observe that the error tends to 0 as $\sigma \ll \rho^3$, even when $\rho = o(1)$.

\section{General lower bound for matrix denoising}\label{sec:lowerbound}
In this section, we will show PCA-denoising is rate-optimal in terms of sample complexity and noise intensity. Theorem 3 of \cite{cai2018rate} has shown the optimality of PCA referring to the recovery of the low-dimensional subspace $V_r$. In its Remark 4, \cite{cai2018rate}  claims that the result can be generalized to lower bounds for matrix denoising. 
Theorem 3 of \cite{montanari2018adapting} has studied the lower bound for matrix denoising under the assumption that $n\asymp d$. Our result below is more general than these two works in some perspectives: we allow $n$ and $d$ to diverge at different speeds (in comparison with \cite{montanari2018adapting}; meanwhile their results don't involve hidden constants), and the estimator $\Xhat$ can be arbitrary (\cite{cai2018rate} assumes $\Xhat$ is of rank at least $r$).

%Both works consider the minimax error, that is, for any estimator $\Xhat$, there exists an $X$ so that $\E[\|\Xhat-X\||X]$ is large. Our result below is different, as $X$ has a fixed distribution and the assumption $n/d = O(1)$ is not required. Instead, we find the lower bound of sample size $n$ to achieve a small estimation error. 

We consider a simplified zigzag model. 
Suppose we only have one segment $v$ and the clean data is $X_i = (t_i + 1)v$ for $t_i \sim Unif(0, 1)$ independently. 
The observed data is $Z_i = X_i + \xi_i$ with noise $\xi_i$. A successful matrix denoising algorithm will generate $\Xhat = (\hat{t}+1)v^T$, which gives the direction $v$ and the points of observation $t_i$, $i \in [n]$. We want to set up the general lower bound $\|\Xhat - X\|_{2, \infty}$ for any estimator $\Xhat$. 

\begin{thm}\label{thm:lowerbound}
Let $v\sim \mathcal{N}(0, d^{-1}I_d)$. Let $t_i \sim Unif(0,1)$ independently, $i \in [n]$. Let $X_i = (t_i + 1)v$. Suppose $Z_i = X_i + \xi_i$ is observed, where $\xi_i \sim \mathcal{N}(0,\sigma^2 I_d)$, $i \in [n]$.
Suppose for some $\epsilon\in (0, d/4)$, the noise level $\sigma>4\epsilon/\sqrt{\Phi(-1)}$ or the sample size $n < d\sigma^2 /(5\epsilon^2)$,  then for any estimator $\Xhat$, there is 
\[
\E_v \E_{t_{[n]},\xi_{[n]}}[\|\Xhat_i-X_i\|_{2, \infty}^2]\geq \epsilon^2, 
\]
where $t_{[n]}:=(t_1,\ldots, t_n), \xi_{[n]}:=(\xi_1,\ldots,\xi_n)$, and $\Phi$ is the cumulative distribution function of standard Gaussian. 
\end{thm}

Theorem \ref{thm:lowerbound} establishes lower bounds for two crucial factors: the signal-to-noise ratio level and the sample size. To achieve a small error rate $\epsilon^2$, it is necessary to have both a low noise level $\sigma = O(\epsilon)$ and a large sample size $n = O(d\sigma^2/\epsilon^2)$. This explains why a spectral gap requirement in the form \eqref{eqn:specgap} is necessary: Using Proposition \ref{prop:zigzag}, we have $\lambda_r(X)\geq c_X \sqrt{n}$ for some constant $c_X$, so \eqref{eqn:specgap} essentially requires $\sigma \sqrt{d}\lesssim \sqrt{n}$, while Theorem \ref{thm:lowerbound} indicates the complementary case will lead to inaccurate denoising result.
When $\sigma = C\epsilon$, the required sample size is $n \geq C d$, aligning with the assumption made in the last estimate in Theorem \ref{thm:linftydavis}. 

%On the other hand, when $\sigma = o(\epsilon)$, our theorem provides a lower bound for the case where $n \ll d$. This specific scenario has not been explored in the existing literature.

%Theorem \ref{thm:linftydavis} demonstrates that the error of Algorithm \ref{alg:denoise} is bounded by $O(\sigma \log n)$ when the condition $d = c_d n$ is satisfied. 

Proof of Theorem \ref{thm:lowerbound} is presented in Section 3 of
\cite{supp}.

\section{Applications of uniform denoising}\label{sec:application}
In this section, we demonstrate the utility of having access to uniformly denoised data points in various downstream statistical learning applications. The main theme is to show that the learning results using uniformly denoised data are close to those achieved by clean data. In this section, we consider the clustering problem and manifold learning using the graphical Laplacian.

We also have results for empirical risk minimization (ERM) in the \supp\citep{supp}. ERM is a widely used approach for training statistical models \citep{hastie2009elements,mei2018landscape}. In ERM, the dataset consists of pairs $(X_i, y_i)$, where $X_i\in \mathbb{R}^d$ and $y_i$ is typically a scalar. The prediction error is measured by a loss function $F(X_i, y_i, \theta)$, where the specific definition of $F$ varies depending on the model. Under mild conditions, we show that the loss function based on uniformly denoised data $\Xhat_i$ is a close approximation of that based on the clean data $X_i$. Since the derivation is rather routine, we leave the details to the Section 4.1 of \cite{supp}.

% \subsection{Regression with RKHS}
% In RKHS, we  have a kernel $k(x,y)$ and we try to obtain the optimal $f$ in $\mathcal{H}$ so that it minimizes the regularized minimal risk
% \[
% L(f)=\frac1n \sum_{i=1}^n|z_i-f(x_i)|^2 +\lambda \|f\|_{\mathcal{H}}^2
% \]
% Here $K_X(i,j)=k(X_i,X_j)$ for some kernels. 
% The representer theorem indicates that the solution will take form 
% \[
% f(x)=f_a(x)=\sum_{i=1}^n a_i k(x,x_i) 
% \]
% for some vectors $a_i$. We can plug this in and find 
% \[
% L(f)=\|K_Xa-z\|^2+\lambda a^TK_Xa
% \]
% So the minimizer is 
% \[
% (K_X^2+\lambda K_X)^{-1}K_X z=(K_X+\lambda I)^{-1} z
% \]

\subsection{Clustering}\label{sec:clustering}
Consider a clustering problem with clean data points $X_{[n]} \in \reals^d$ in $K$ clusters \citep{omran2007overview}. Denote $\calC_k$ as the index set of data points in the $k$-th cluster and the division is $\calC = \{\calC_1, \cdots, \calC_K\}$. $K$-means in \cite{kmeans} is a very popular clustering algorithm. It aims to find the labels $\calChat$ and centers $\mhat = \mhat_{[K]}$, so that the following is minimized
\begin{equation}
\label{eqn:kmeanloss}
f(\calC, m) =\sum_{k\in [K]} \sum_{i \in \calC_k} \|X_{i}-m_{k}\|^2.
\end{equation}
When the clean data $X$ is available, $k$-means will achieve a good clustering result. 

In practice, we only have access to noisy data $Z_i=X_i+\xi_i$. If we simply replace $X_i$ in \eqref{eqn:kmeanloss} with $Z_i$, it is unlikely that we can get a good estimation. Using denoised data matrix $\Xhat$, we would consider minimizing the loss function  
\begin{equation}
\label{eqn:kmeanlossdenoise}
\fhat(\calC, m) = \sum_{k\in [K]} \sum_{i \in \calC_k} \|\Xhat_{i}-m_{k}\|^2,
\end{equation}
where $\Xhat_i$ comes from Algorithm \ref{alg:denoise}. We expect the clustering result from $\fhat$ to have a performance similar to that of using $X$ directly.
%Like in the applications with ERM, we expect the clustering result from $\fhat$ will have performance similar to the one training using clean data. 
This is similar to 
existing consistency analysis of clustering in \cite{pollard1981strong,rakhlin2006stability}, and \cite{ bubeck2009nearest}, where the loss function convergence is of main interest.

To evaluate the clustering performance, the weak consistency focuses on the clustering error rate and the strong consistency focuses on the number of clustering errors. With the uniform error bound between $\Xhat_i$ and $X_i$, we can control the clustering error on each data point and conclude a strong consistency result in Proposition \ref{cor:cluster}. Such results cannot be achieved by the error control on the Frobenius norm.

On the other hand, without further assumptions, it can be hard to show label consistency. This is mainly because $k$-means clustering result is non-unique and can have discontinuous dependence at data points near the boundary between two clusters. We remove such cases by
further assuming there is a center $m_k \in \reals^d$ for the $k$-th cluster, so that data points in $\calC_k$ are close to $m_k$. Such a requirement is standard in $k$-means related literature (\cite{kmeans, jin2016influential, jin2017phase}).
A rigorous assumption is as follows.
\begin{aspt}\label{aspt:clustering}
Suppose that under the division $\calC$, the clusters satisfy the following with absolute constants $c_0,\delta,c_m$:
\begin{itemize}
    \item The cluster size does not degenerate: $|\calC_k|\geq c_0n$ for all $k \in [K]$.
    \item Data points in each cluster are close to a center: $\|X_i-m_k\|\leq \delta,\forall i\in \calC_k$.
    \item The cluster centers are far apart: $\|m_j-m_k\|\geq c_m$, for all $1 \leq j\neq k \leq K$.
\end{itemize}
\end{aspt}

%We assume the $k$-th cluster has a center $m_k$, so that data points in $\calC_k$ is closer to $m_k$ than other centers. 
% In other words, there is $\delta>0$ so that with probability $1 - O(1/n)$ and any $k \in [K]$, 
% \begin{equation}
% \label{eqn:kmeansetup}
% \|X_{i}-m_k\|\leq \delta+\min_{j\neq k}\|X_i-m_j\|,\quad \forall i\in \calC_k. 
% \end{equation}

% One caveat of k-means is that the solution is not unique or continuous w.r.t. the data point distribution. A simple example that can illustrate is to consider clustering three data points $x_1=1,x_2=2, x_3=3$ into $K=2$ classes. It is easy to see if there are two optimal solutions, where the centers are $m_1=1.5, m_2=3$, or $m_1=1, m_2=2.5$. And if have a small perturbation on $x_2$ to be $x_2=2+\epsilon$, then depending on the sign of $\epsilon$, one of the two solutions will cease to be optimal, while the other one is perturbed. 

% To avoid such singular scenarios, we assume the data labeling is unique. That is 
% \[
% l(\pi^*)-l(\hat{\pi})\geq \pi-
% \]

%We assume

\begin{prop}
\label{cor:cluster}
Consider a clustering problem with noisy data points $Z_{[n]}$ and underlying truth $\calC$. Suppose $\err{\Xhat - X} \leq \epsilon$. Let $\calChat$ be the minimizer of the loss function $\fhat(\calC, m)$ defined in \eqref{eqn:kmeanlossdenoise}. Then 
\begin{itemize}
    \item $\calChat$ almost minimize the loss function with clean data $f(\calC, m), 
    $\[
f(\calChat, \mhat)\leq f(\calC, m)+4\epsilon.
\]
\item Furthermore, suppose Assumption \ref{aspt:clustering} holds with $c_m> 2(\delta+\epsilon+2\sqrt{2(\delta^2+\epsilon^2)/c_0})$, then $k$-means gives a perfect clustering result, i.e. the division $\calChat$ is exactly the same with $\calC$. 
\end{itemize}
% Then if 
% \[
% c_m> 2(\sigma+\epsilon+2\sqrt{(\sigma^2+\epsilon^2)/c_0}). 
% \]
% The clustering results are identical, i.e. there is $\pi:[K]\to [K]$, so that $X$
\end{prop}
The proof can be found in the \supp\citep{supp}.

\subsection{Manifold learning with graphical Laplacian}\label{sec:manifold}

% Another important unsupervised task is 
% feature extraction. It  involves identifying and extracting meaningful features from raw data. The goal is to reduce the dimensionality of the data while preserving its essential characteristics. In fact, the PCA approach this paper investigated is an important feature extraction method.

% Manifold learning is an important tool of feature extraction. It tries to obtain a low-dimensional representation of high-dimensional data that captures the underlying geometric structure. It can be used to visualize high-dimensional data in two or three dimensions, making it easier to understand and interpret.

% Mathematically, we assume there is a low dimensional manifold $M$ lies in a $d$ dimensional ambient space, and we try to obtain low dimensional feature maps using data points from $M$.

Graphical Laplacian is an important data processing tool in various applications. It is commonly used in spectral clustering, dimensionality reduction, and graph-based machine learning tasks that involve analyzing data on nonlinear manifolds. It enables users to capture the underlying graphical structures \citep{von2007tutorial,singer2006graph, khoo2024temporal}. To illustrate, consider a kernel function  $k(\cdot, \cdot)$ as follows: 
\begin{aspt}
\label{aspt:kernel}
The function $k(x,y)$ is symmetric and Lipschitz in both $x$ and $y$. Furthermore, there is an absolute constant $K > 0$ so that $1/K\leq k(x,y)\leq K$ when $\|x\|,\|y\|\leq 1$.
\end{aspt}
A similar version of this assumption can be found in \cite{von2008consistency}. A quick example that satisfies Assumption \ref{aspt:kernel} is the most commonly used Gaussian kernel, where $k(x,y)=\exp(-\|x-y\|^2/(2b^2))$.

Suppose the clean data $X_i$ are i.i.d. samples from a distribution $p$ on the unit ball $\{x:\|x\|\leq 1\}$. The population normalized graphical Laplacian operator is defined as 
\[
\mathcal{L}f(x)=\int (1-\frac{k(x,y)}{\sqrt{d(x)d(y)}})f(y)p(y)dy,\quad \text{where } d(x)=\int k(x,y)p(y)dy. 
\]
$\mathcal{L}$ can be used for different purposes. It can directly be used to optimize a certain utility function. Its eigenfunctions can be used to extract a low-dimensional representation of the manifold or to establish spectral clustering results.

In practice, we do not have access to $\mathcal{L}$ directly. We first build the empirical operator based on the clean data $X_i$. Define the kernel matrix $A$ where $A_{i,j} = k(X_i, X_j)$, and the diagonal matrix $D$ with $D_{i,i} = \sum_{j=1}^n A_{i,j}$. The normalized Laplacian is given by 
\begin{equation}
\label{eqn:gL}
L = I-D^{-1/2}AD^{-1/2}.
\end{equation}
%The normalized Laplacian is given by  $L=(I-D^{-1/2}AD^{-1/2})$. 
Under mild conditions, \cite{von2008consistency} has shown $L$ converge to $\mathcal{L}$ as an operator, and so are its spectrum and eigenvalues. 
In particular, with $k$ as a Gaussian kernel, if $\lambda_0$ is a simple eigenvalue  of $\mathcal{L}$ with its eigen-function $\|u(x)\|_\infty=1$, then Theorems 16 and 19 of \cite{von2008consistency} indicate that $L$ has an eigenvalue $\lambda$ close to $\lambda_0$, and the corresponding  eigenvector $v$ (of which the norm may not be $1$) satisfies that 
\[
\max_i|v_i-u(X_i)|=O(n^{-1/2}).
\]
Hence, $v$ will carry similar information as $u$ and can be used for further analysis such as clustering. 
These results provide a theoretical foundation for utilizing the eigenvectors of $L$ in statistical inference. 
%In particular, if $u$ carries clustering information, then one can do clustering based on $v$. 

In the context of this paper, it is natural to ask whether such a result can be extended to the case only $Z_i=X_i+E_i$ are observed. Intuitively, we will first apply Algorithm \ref{alg:denoise} to obtain the denoised data $\Xhat_i$, and achieve $\Lhat$ correspondingly. Define $\Ahat$ where $\Ahat_{i,j} = k(\Xhat_i,\Xhat_j)$, $\Dhat$ where $\Dhat_{i,i}=\sum_{j=1}^n\Ahat_{i,j}$, and 
\begin{equation}
\hat{L}=(I-\Dhat^{-1/2}\Ahat \Dhat^{-1/2}). 
\end{equation}
We are interested in how different $\Lhat$ would be from $L$, in the sense of matrix $\ell_{\infty}$ norm, the informative eigenvalues, and the corresponding eigenvectors.  The following proposition addresses the result:
\begin{prop}
\label{prop:lap}
Consider a kernel function that Assumption \ref{aspt:kernel} holds. Suppose the denoised matrix follows $\err{\Xhat - X} \leq \epsilon$. Then there is a constant $C > 0$ independent of $n$ and $d$ so that in the $\ell_\infty$ operator-norm
\[
\|\Lhat-L\|_\infty\leq C\epsilon. 
\]
Furthermore, suppose $\mathcal{L}$ has a simple eigenvalue $\lambda_0\neq 1$. Then there are $N_0$, $\epsilon > 0$, and a constant $C > 0$ independent of $n$ and $d$, so that when $n>N_0$, $\epsilon<\epsilon_0$, $L$,$\Lhat$ has eigen-pair $(v,\lambda) $ and $(\hat{v},\hat{\lambda})$ respectively, so that  
\[
|\hat{\lambda}-\lambda|\leq C\epsilon,\quad
\|a\hat{v}-v\|_\infty\leq C\epsilon \|v\|_\infty, 
\]
where $a=1$ or $-1$. 
\end{prop}

%%%%%%%%%%%%%%%%%%%%%%%%%%%%%%%%%%%%%%%
%%%%%%%%%%%%%%%%%%%%%%%%%%%%%%%%%%%%%%%
%%%%%%%%%%%%%%%%%%%%%%%%%%%%%%%%%%%%%%%
\section{Numerical simulation}
\label{sec:simulation}

To support our theoretical results, we provide some numerical examples in this section. 
We consider a two-class clustering setting where the clean data are sampled from two separate low-dimensional zigzag lines embedded in a high-dimensional space. In detail, we first generate two zigzag lines, each consisting of 10 segments within a unit ball in $\mathbb{R}^3$, as seen in Figure \ref{fig: zig zag 3D}(a). 
%They are separated in the sense that $y_3 > 0.1$ for one zigzag line and $z_3<-0.1$ for another. 
Then we embedded them into $d$-dimensional space by an orthonormal matrix $\Omega \in \reals^{d \times 3}$ with $d$ being some large number to be specified later. Denote the embedded zigzag lines as $x_0(t)$ and $x_1(t)$ for $0 \leq t \leq 1$. 
The clean data are generated as 
$X_i = \Omega \, x_{\ell(i)}(t_i)$, where $\ell(i) \sim Bernoulli(1/2)$ and 
$t_i \sim U[0,1]$, $i \in [n]$. Here $\ell(i)$ can be regarded as the label of this sample and our goal is to recover it. 
Finally, we add the noise to the clean data, where $Z_i = X_i + \xi_i$, $\xi_i \sim \mathcal{N}(0, \sigma^2 I_d)$. The algorithm will be applied to the noisy data $Z$. 
%Then i.i.d. noise $\xi_i\sim \mathcal{N}(0,\sigma^2 I_{d})$ are added to formulate the noisy data $Z_i=X_i+\xi_i$.  

We first check the uniform error of PCA-denoised data by Algorithm \ref{alg:denoise}. 
The error in Theorem \ref{thm:linftydavis} requires the information of $\lambda_r(X)$, which is shown to be $O(\sqrt{n})$ in Proposition \ref{prop:zigzag} for the zigzag scenario. 
According to the generation procedure, the clean data $X$ lies in the space with a dimension of $r = 3$, so we calculate $\lambda_3(X)$ for $n$ ranging from $10^2$ to $70^2$. The embedding $\Omega$ does not change the eigenvalue, and $\lambda_3(X) = \lambda_3(x)$ where $x = (x_{\ell(1)}(t_1), \cdots, x_{\ell(n)}(t_n))$. We calculate it based on $x$ so the dimension $d$ does not affect the result. 
As seen in Figure \ref{fig: zig zag 3D}(b), $\lambda_3(X)$ scales linearly with $\sqrt{n}$, which follows Proposition \ref{prop:zigzag}. Such a result indicates the simplified uniform error in Theorem \ref{thm:linftydavis} holds, that $\err{X - \Xhat} = O(\sigma + \sigma^4\log n)$ when $d = c_d n$. 
%Indeed, if we generate denoised data $\Xhat_i$ using Algorithm \ref{alg:denoise}, $\|\Xhat-X\|_{2,\infty}$ will be independent of $d$  and scale linearly with $\sigma$, as long as the sample size $n$ is a proportion of $d$. 
We fix $n = 0.1d$ and let $d$ range from $10^2$ to $10^4$, where the added noise $\sigma \in \{0.0025,0.05,0.1\}$. 
%To illustrate this, we have simulated the denoising process with $\sigma=0.0025,0.05,0.1$ and $d$ ranging from $10^2$ to $10^4$. We fix $n=0.1d$ in the process. 
Figure \ref{fig: zig zag 3D}(c) presents the uniform error of the noisy data $Z$ and the denoised data $\Xhat$ by Algorithm \ref{alg:denoise}. It can be found that $\|\Xhat-X\|_{2,\infty}/\sigma$ stays constant throughout the regime while $\|Z-X\|_{2,\infty}/\sigma$ grows as $\sqrt{d}$. Even in the high-dimensional case that $d = 10n$, the PCA-denoising algorithm can reduce the curse of dimension and reduce the uniform error to be $O(\sigma)$. 

% \begin{figure}[ht!]

%      % \hfill
%      \begin{subfigure}[b]{0.5\textwidth}
%          \centering
%  \includegraphics[width=\textwidth]{figs/zig_zag_3D.png}
%          \caption{}
%          \label{fig: zig zag 3D}
%      \end{subfigure}
%      \hfill
%      \begin{subfigure}[b]{0.45\textwidth}
%          \centering
%  \includegraphics[width=\textwidth]{figs/minimal_sigval.png}
%          % \caption{Landscape of the objective function evaluated along a line starting from the global minimum in a random direction. The parameter $t$ denotes the $l^2$ distance of the evaluated points from the global minimum. }
%          \caption{}
%          \label{fig: minimal sigval}
%      \end{subfigure}
%      \caption{(a) The generated low-dimensional dataset $X_l$ containing two clusters. Each cluster is located on a zigzag line within the unit ball. The two zigzag lines are separated by parallel hyperplanes with a distance of $0.2$.
% (b) The $r$-th singular value of the clean dataset $X \in \mathbb{R}^{n \times d}$ or $X_l\in \mathbb{R}^{n \times r}$ scales linearly with the square root of the sample size $n$.}
% \label{fig: local_landscape}
% \end{figure}

\begin{figure}[ht!]
    \centering
\includegraphics[width=.75\textwidth]{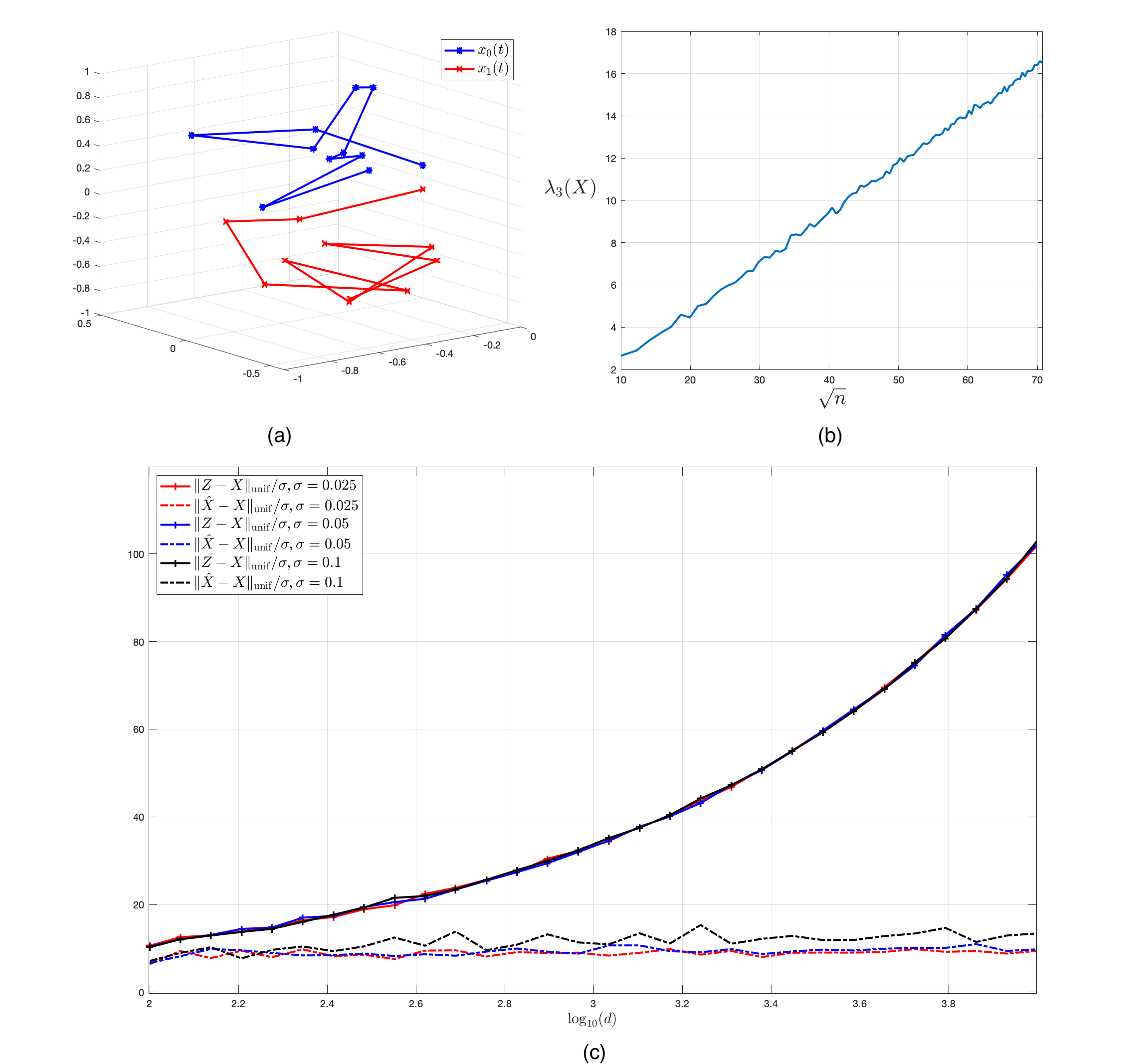}
    \caption{Properties of data generated from two zig-zag lines. (a) The two zigzag lines where the clean data is generated from. They are separated by parallel hyperplanes with a distance of $0.2$. 
    %The 3D scatter plot of  generated low-dimensional dataset $x_{l(i)}(t_i)$ containing two clusters. Each cluster is located on a zigzag line within the unit ball. 
    (b) The third singular value of the clean data $X \in \mathbb{R}^{n \times 3}$, which scales linearly with the square root of the sample size $n$. (c) The normalized uniform errors of the noisy data $Z$, denoised data $\hat X$ achieved at different dimension $d.$
    With a fixed standard deviation of the noise $\sigma$, the uniform errors obtained from the noisy data $Z \in \mathbb{R}^{n \times d}$ increase with the dimension $d$ and sample size $n=0.1d$. However, the uniform errors obtained from the denoised data $\hat X$ remain approximately constant. The results are consistent with different standard deviations, $\sigma=0.025, 0.05,$ and $0.1$.}
    \label{fig: zig zag 3D}
\end{figure}

%This is in line with the statements of Theorem \ref{thm:linftydavis}.  

We then investigate the denoising effects on downstream statistical inference, particularly in the graphical Laplacian eigenvectors and spectral clustering. 
Let $L(B)$ denote the graphical Laplacian of any data matrix $B$. We find the eigenvector corresponding to the second smallest eigenvalue, which is called the {\it Fiedler eigenvector}. Denote it as $\eta_2(B)$. 
It is well known that the component signs of $\eta_2(B)$ can be used to infer cluster labels (\cite{von2008consistency}). 
We are interested in the error of the Fiedler eigenvectors from different data sources and the associated clustering results. 

To illustrate the importance of the uniform error, we consider three data matrices in this study: the clean data $X$, the PCA-denoised data $\Xhat$, and a new data set $\Xtilde$ with errors not uniformly distributed. To make the comparison fair, the new matrix $\Xtilde$ has the same average error as $\Xhat$,  i.e.
\begin{equation}
\label{tmp:eqavg}
\frac{1}{n}\sum_{i=1}^n \|X_i-\Xhat_i\|^2=\frac{1}{n}\sum_{i=1}^n \|X_i-\Xtilde_i\|^2.
\end{equation}
We design $\Xtilde$ in a way that some samples have a large error and other samples have no errors. In detail, we randomly select 10\% of indices as $\mathcal{I}$, and define $\Xtilde_i$ as: 
\[
\Xtilde_i=\begin{cases}X_i,\quad &\text{if }i\in \mathcal{I}^c;\\
\alpha X_i+(1-\alpha)Z_i,\quad &\text{if }i\in \mathcal{I},
\end{cases}
\]
where $\alpha\in (0,1)$ is a constant that ensures \eqref{tmp:eqavg}.  
In other words, 90\% of data in $\Xtilde$ are clean; the others have $i.i.d.$ errors as $(1-\alpha)\xi_i$, where $\alpha$ is chosen so that the average error is the same with $\Xhat$. 
Comparisons between $\eta_2(\Xhat)$ and $\eta_2(\Xtilde)$ explain how uniform error is important in practice. 

\begin{figure}[ht!]
    \centering
\includegraphics[width = .9\textwidth]{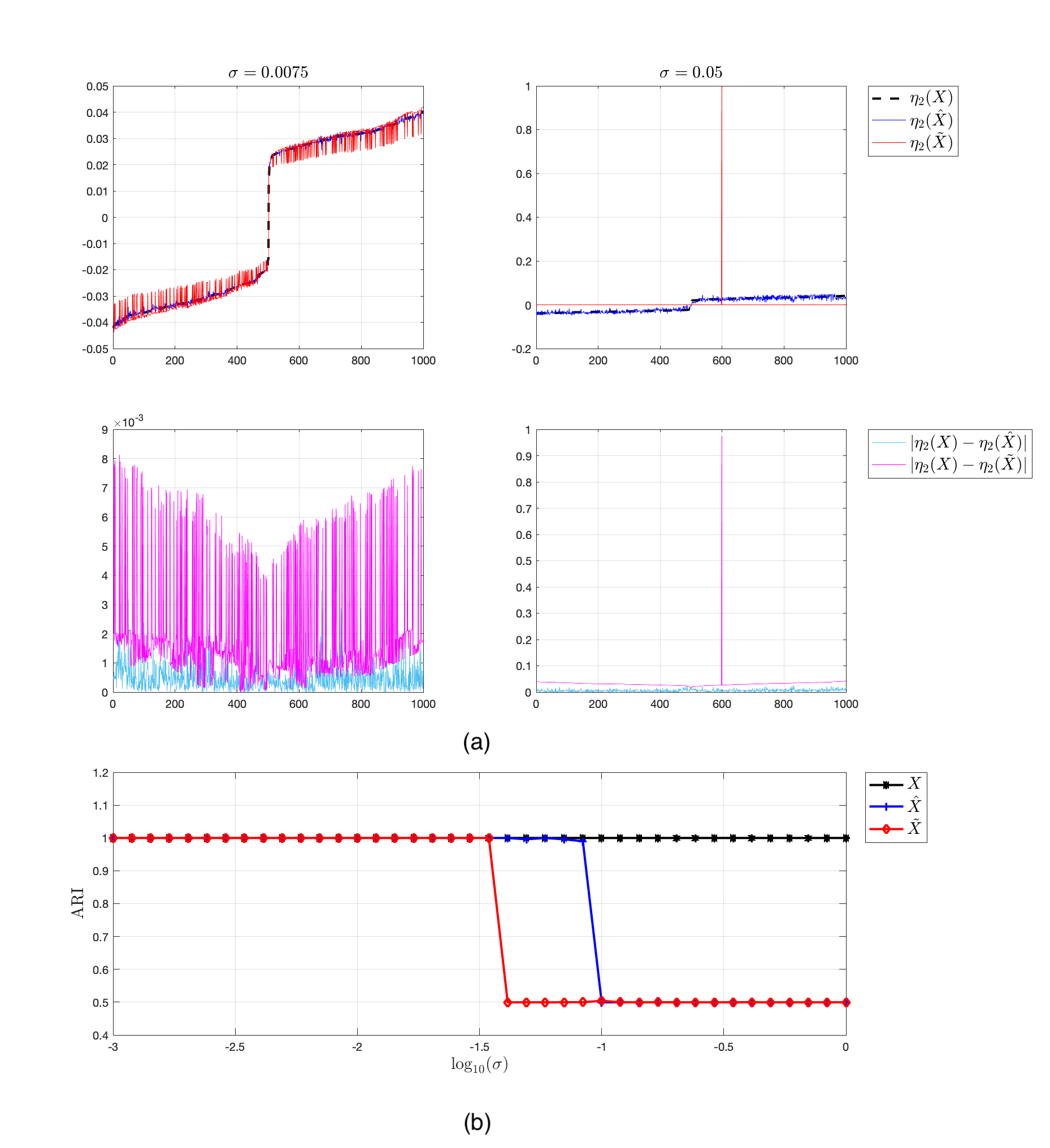}
    \caption{Spectral clustering with different data sources. (a) The Fiedler eigenvectors of the Laplacian formed by the clean data, the denoised data, and the "averaged denoised" data at two noise levels. 
    %Denoted as $\eta_2(X)$, $\eta_2(\hat X)$ and $\eta_2(\tilde X)$, respectively. 
    The first row plots the second (Fiedler) eigenvector of the graphical Laplacian obtained by the clean data $X$,  the denoised data $\hat X$, and averaged denoised data $\tilde X$. They are denoted by $\eta_2(X), \eta_2(\hat X)$, and $\eta_2(\tilde X)$, respectively. The second row plots the pointwise absolute difference between $(\eta_2(X), \eta_2(\hat X))$ and $(\eta_2(X), \eta_2(\tilde X))$. The standard deviation of the noise is chosen to be $\sigma=0.0075$ and $0.05$, corresponding to the two columns. (b) The adjusted rand indices (ARI) by spectral clustering using clean data $X$, denoised data $\hat X$, and the averaged denoised data $\tilde X$ when $\sigma$ changes.}
    \label{fig: spect clust}
\end{figure}

We fix $d=5000$ and $n=1000$. The bandwidth for graphical Laplacian is chosen as $b=\sqrt{0.005}$. 
We first compare the Fiedler eigenvectors from three data matrices $X$, $\Xhat$, and $\Xtilde$ at two noise levels $\sigma = 0.0075$ and $\sigma = 0.05$, shown in Figure \ref{fig: spect clust}(a). We adjust the order of data samples so that samples from $x_0(t)$ come first and those from $x_1(t)$ are all at last.  In the top panel of Figure \ref{fig: spect clust}(a), we can see the sign $\eta_2(X)$ meets the label $\ell$. 
For the two data with small noise $\sigma=0.0075$, $\eta_2(\Xhat)$ and $\eta_2(\Xtilde)$ are close to $\eta_2(X)$ at a large scale. The error of $\eta_2(\Xhat)$ is uniformly small but $\eta_2(\Xtilde)$ has spiked errors. This is consistent with Proposition \ref{prop:lap}. When $\sigma$ increases from $0.0075$ to $0.05$, the spiked error in $\|\Xtilde_i - X_i\|$ is so large that the distance-based clustering algorithm fails. 
Meanwhile, the  Fiedler eigenvector obtained by $\Xhat$ can still preserve the correct cluster labels. 
We then study the clustering results of the three data matrices for $\sigma$ ranging from $10^{-3}$ to $1$, shown in 
Figure \ref{fig: spect clust}(b). We use the adjusted rand index (ARI) to measure the clustering accuracy, where $ARI = 1$ indicates a perfect clustering and $ARI = 0.5$ means a random guess. 
The performance of clustering using other data sources would of course depend on the strength of noise. Not surprisingly, because of the crash in estimating $\eta_2(X)$, using $\Xtilde$ fails when $\sigma \ge 0.05$ and using $\Xhat$ gives a satisfactory clustering result till $\sigma \ge 0.1$. 

\subsection*{Acknowledgments}
The authors would like to thank Yuehaw Khoo for the throughout discussion and reading of this project. The authors also would like to thank the two anonymous referees, an Associate Editor and the Editor for their comments that improved the quality of this paper. 

\subsection*{Funding}
Xin T. Tong's research is supported by Singapore MOE grant Tier-1-A-8000459-00-00. Wanjie Wang's research is supported by Singapore MOE grant Tier-1-A-8001451-00-00. 
Yuguan Wang's research is supported by the Department of Statistics, University of Chicago.

\bibliographystyle{plainnat}
\bibliography{denoising.bib}

\begin{thebibliography}{59}
\providecommand{\natexlab}[1]{#1}
\providecommand{\url}[1]{\texttt{#1}}
\expandafter\ifx\csname urlstyle\endcsname\relax
  \providecommand{\doi}[1]{doi: #1}\else
  \providecommand{\doi}{doi: \begingroup \urlstyle{rm}\Url}\fi

\bibitem[Abbe and Sandon(2015)]{abbe2015community}
Emmanuel Abbe and Colin Sandon.
\newblock Community detection in general stochastic block models: Fundamental
  limits and efficient algorithms for recovery.
\newblock \emph{In the proc. of IEEE FOCS}, pages 670--688, 2015.
\newblock \doi{https://doi.org/10.1109/FOCS.2015.47}.

\bibitem[Abbe et~al.(2020)Abbe, Fan, Wang, and Zhong]{abbe2020entrywise}
Emmanuel Abbe, Jianqing Fan, Kaizheng Wang, and Yiqiao Zhong.
\newblock Entrywise eigenvector analysis of random matrices with low expected
  rank.
\newblock \emph{Ann. Stat.}, 48\penalty0 (3):\penalty0 1452--1474, 2020.
\newblock \doi{https://doi.org/10.1214/19-AOS1854}.

\bibitem[Abdi and Williams(2010)]{abdi2010principal}
Herv{\'e} Abdi and Lynne~J Williams.
\newblock Principal component analysis.
\newblock \emph{Wiley Interdiscip. Rev. Comput. Stat.}, 2\penalty0
  (4):\penalty0 433--459, 2010.
\newblock \doi{https://doi.org/10.1002/wics.101}.

\bibitem[Bao et~al.(2021)Bao, Ding, and Wang]{bao2021singular}
Zhigang Bao, Xiucai Ding, and Ke~Wang.
\newblock Singular vector and singular subspace distribution for the matrix
  denoising model.
\newblock \emph{Ann. Stat.}, 49\penalty0 (1):\penalty0 370--392, 2021.
\newblock \doi{https://doi.org/10.1214/20-AOS1960}.

\bibitem[Bhatia(2013)]{bhatia2013matrix}
Rajendra Bhatia.
\newblock \emph{Matrix Analysis}, volume 169.
\newblock Springer Science \& Business Media, 2013.
\newblock \doi{http://dx.doi.org/10.1007/978-1-4612-0653-8}.

\bibitem[Brunton and Kutz(2022)]{brunton2022data}
Steven~L Brunton and J~Nathan Kutz.
\newblock \emph{Data-driven Science and Engineering: Machine Learning,
  Dynamical Systems, and Control}.
\newblock Cambridge University Press, 2022.
\newblock \doi{https://doi.org/10.1017/9781108380690}.

\bibitem[Bubeck and Luxburg(2009)]{bubeck2009nearest}
S{\'e}bastien Bubeck and Ulrike~von Luxburg.
\newblock Nearest neighbor clustering: A baseline method for consistent
  clustering with arbitrary objective functions.
\newblock \emph{J. Mach. Learn. Res.}, 10:\penalty0 657--698, 2009.
\newblock \doi{http://jmlr.org/papers/v10/bubeck09a.html}.

\bibitem[Cai and Zhang(2018)]{cai2018rate}
T~Tony Cai and Anru Zhang.
\newblock Rate-optimal perturbation bounds for singular subspaces with
  applications to high-dimensional statistics.
\newblock \emph{Ann. Stat.}, 46\penalty0 (1):\penalty0 60--89, 2018.
\newblock \doi{https://doi.org/10.1214/17-AOS1541}.

\bibitem[Cape et~al.(2019{\natexlab{a}})Cape, Tang, and Priebe]{twoinifinity1}
Joshua Cape, Minh Tang, and Carey~E Priebe.
\newblock The two-to-infinity norm and singular subspace geometry with
  applications to high dimensional statistics.
\newblock \emph{Ann. Stat.}, 47\penalty0 (5):\penalty0 2405--2439,
  2019{\natexlab{a}}.
\newblock \doi{https://doi.org/10.1214/18-AOS1752}.

\bibitem[Cape et~al.(2019{\natexlab{b}})Cape, Tang, and Priebe]{twoinifinity2}
Joshua Cape, Minh Tang, and Carey~E Priebe.
\newblock Signal-plus-noise matrix models: eigenvector deviations and
  fluctuations.
\newblock \emph{Biometrika}, 106\penalty0 (1):\penalty0 243--250,
  2019{\natexlab{b}}.
\newblock \doi{https://doi.org/10.1093/biomet/asy070}.

\bibitem[Chen et~al.(2021)Chen, Chi, Fan, Ma, et~al.]{chen2021spectral}
Yuxin Chen, Yuejie Chi, Jianqing Fan, Cong Ma, et~al.
\newblock Spectral methods for data science: A statistical perspective.
\newblock \emph{Found. Trends Mach. Learn.}, 14\penalty0 (5):\penalty0
  566--806, 2021.
\newblock \doi{http://dx.doi.org/10.1561/2200000079}.

\bibitem[Cressie and Wikle(2015)]{cressie2015statistics}
Noel Cressie and Christopher~K Wikle.
\newblock \emph{Statistics for Spatio-temporal Data}.
\newblock John Wiley \& Sons, 2015.
\newblock ISBN ISBN: 978-0-471-69274-4.

\bibitem[Davis and Kahan(1970)]{daviskahan}
Chandler Davis and William~Morton Kahan.
\newblock The rotation of eigenvectors by a perturbation. iii.
\newblock \emph{SIAM J. Numer. Anal.}, 7\penalty0 (1):\penalty0 1--46, 1970.
\newblock \doi{https://doi.org/10.1137/0707001}.

\bibitem[Ding(2020)]{ding2020high}
Xiucai Ding.
\newblock High dimensional deformed rectangular matrices with applications in
  matrix denoising.
\newblock \emph{Bernoulli}, 26\penalty0 (1):\penalty0 387--41, 2020.
\newblock \doi{https://doi.org/10.3150/19-BEJ1129}.

\bibitem[Dong and Tong(2021)]{dong2021replica}
Jing Dong and Xin~T Tong.
\newblock Replica exchange for non-convex optimization.
\newblock \emph{J. Mach. Learn. Res.}, 22\penalty0 (1):\penalty0 7826--7884,
  2021.
\newblock \doi{https://www.jmlr.org/papers/v22/20-697.html}.

\bibitem[Donoho and Gavish(2014)]{donoho2014minimax}
David Donoho and Matan Gavish.
\newblock Minimax risk of matrix denoising by singular value thresholding.
\newblock \emph{Ann. Stat.}, 42\penalty0 (6):\penalty0 2413--2440, 2014.
\newblock \doi{DOI: 10.1214/14-AOS1257}.

\bibitem[Donoho et~al.(2018)Donoho, Gavish, and Johnstone]{donoho2018optimal}
David~L Donoho, Matan Gavish, and Iain~M Johnstone.
\newblock Optimal shrinkage of eigenvalues in the spiked covariance model.
\newblock \emph{Ann. Stat.}, 46\penalty0 (4):\penalty0 1742--1778, 2018.
\newblock \doi{https://doi.org/10.1214/17-AOS1601}.

\bibitem[Donoho et~al.(2000)]{donoho2000high}
David~L Donoho et~al.
\newblock High-dimensional data analysis: The curses and blessings of
  dimensionality.
\newblock \emph{AMS math challenges lecture}, 1\penalty0 (2000):\penalty0
  1--32, 2000.
\newblock \doi{https://api.semanticscholar.org/CorpusID:5293263}.

\bibitem[Fan et~al.(2018)Fan, Wang, and Zhong]{fan2018eigenvector}
Jianqing Fan, Weichen Wang, and Yiqiao Zhong.
\newblock An $l_\infty$ eigenvector perturbation bound and its application to
  robust covariance estimation.
\newblock \emph{J. Mach. Learn. Res.}, 18\penalty0 (207):\penalty0 1--42, 2018.
\newblock \doi{https://jmlr.org/papers/volume18/16-140/16-140.pdf}.

\bibitem[Hartigan and Wong(1979)]{kmeans}
John~A Hartigan and Manchek~A Wong.
\newblock Algorithm as 136: A k-means clustering algorithm.
\newblock \emph{J. R. Stat. Soc. Ser. C Appl. Stat.}, 28\penalty0 (1):\penalty0
  100--108, 1979.
\newblock \doi{http://dx.doi.org/10.2307/2346830}.

\bibitem[Hastie et~al.(2009)Hastie, Tibshirani, Friedman, and
  Friedman]{hastie2009elements}
Trevor Hastie, Robert Tibshirani, Jerome~H Friedman, and Jerome~H Friedman.
\newblock \emph{The Elements of Statistical Learning: Data Mining, Inference,
  and Prediction}, volume~2.
\newblock Springer, 2009.

\bibitem[Hu and Wang(2024)]{hu2023network}
Y~Hu and W~Wang.
\newblock {Network-adjusted covariates for community detection}.
\newblock \emph{Biometrika}, page asae011, 02 2024.
\newblock \doi{10.1093/biomet/asae011}.

\bibitem[Jin and Wang(2016)]{jin2016influential}
Jiashun Jin and Wanjie Wang.
\newblock Influential features pca for high dimensional clustering.
\newblock \emph{Ann. Stat.}, 44\penalty0 (6):\penalty0 2323--2359, 2016.
\newblock \doi{DOI: 10.1214/15-AOS1423}.

\bibitem[Jin et~al.(2017)Jin, Ke, and Wang]{jin2017phase}
Jiashun Jin, Zheng~Tracy Ke, and Wanjie Wang.
\newblock Phase transitions for high dimensional clustering and related
  problems.
\newblock \emph{Ann. Stat.}, 45\penalty0 (5):\penalty0 2151--2189, 2017.
\newblock \doi{DOI: 10.1214/16-AOS1522}.

\bibitem[Johnstone and Lu(2009)]{johnstone2009consistency}
Iain~M Johnstone and Arthur~Yu Lu.
\newblock On consistency and sparsity for principal components analysis in high
  dimensions.
\newblock \emph{JASA}, 104\penalty0 (486):\penalty0 682--693, 2009.
\newblock \doi{https://doi.org/10.1198/jasa.2009.0121}.

\bibitem[Jolliffe(2005)]{jolliffe2005principal}
Ian Jolliffe.
\newblock Principal component analysis.
\newblock \emph{Encyclopedia of Statistics in Behavioral Science}, 2005.
\newblock \doi{https://doi.org/10.1002/0470013192.bsa501}.

\bibitem[Jolliffe(1972)]{jolliffe1972discarding}
Ian~T Jolliffe.
\newblock Discarding variables in a principal component analysis. i: Artificial
  data.
\newblock \emph{J. R. Stat. Soc. Ser. C Appl. Stat.}, 21\penalty0 (2):\penalty0
  160--173, 1972.
\newblock \doi{https://doi.org/10.2307/2346488}.

\bibitem[Khoo et~al.(2024)Khoo, Tong, Wang, and Wang]{khoo2024temporal}
Yuehaw Khoo, Xin~T Tong, Wanjie Wang, and Yuguan Wang.
\newblock Temporal label recovery from noisy dynamical data.
\newblock \emph{arXiv preprint arXiv:2406.13635}, 2024.
\newblock \doi{https://doi.org/10.48550/arXiv.2406.13635}.

\bibitem[Leeb(2021)]{leeb2021matrix}
William~E Leeb.
\newblock Matrix denoising for weighted loss functions and heterogeneous
  signals.
\newblock \emph{SIMODS}, 3\penalty0 (3):\penalty0 987--1012, 2021.
\newblock \doi{https://doi.org/10.1137/20M1319577}.

\bibitem[Mei et~al.(2018)Mei, Bai, and Montanari]{mei2018landscape}
Song Mei, Yu~Bai, and Andrea Montanari.
\newblock The landscape of empirical risk for nonconvex losses.
\newblock \emph{Ann. Stat.}, 46\penalty0 (6A):\penalty0 2747--2774, 2018.
\newblock \doi{https://www.jstor.org/stable/26542881}.

\bibitem[Menard(2002)]{menard2002applied}
Scott Menard.
\newblock \emph{Applied Logistic Regression Analysis}.
\newblock Number 106. Sage, 2002.

\bibitem[Montanari et~al.(2018)Montanari, Ruan, and Yan]{montanari2018adapting}
Andrea Montanari, Feng Ruan, and Jun Yan.
\newblock Adapting to unknown noise distribution in matrix denoising.
\newblock \emph{arXiv preprint arXiv:1810.02954}, 2018.
\newblock \doi{https://doi.org/10.48550/arXiv.1810.02954}.

\bibitem[Nadakuditi(2014)]{nadakuditi2014optshrink}
Raj~Rao Nadakuditi.
\newblock Optshrink: An algorithm for improved low-rank signal matrix denoising
  by optimal data-driven singular value shrinkage.
\newblock \emph{IEEE Trans. Inf. Theory}, 60\penalty0 (5):\penalty0 3002--3018,
  2014.
\newblock \doi{10.1109/TIT.2014.2311661}.

\bibitem[Omran et~al.(2007)Omran, Engelbrecht, and Salman]{omran2007overview}
Mahamed~GH Omran, Andries~P Engelbrecht, and Ayed Salman.
\newblock An overview of clustering methods.
\newblock \emph{Intell. Data Anal.}, 11\penalty0 (6):\penalty0 583--605, 2007.
\newblock \doi{10.3233/IDA-2007-11602}.

\bibitem[O’Rourke et~al.(2024)O’Rourke, Vu, and Wang]{o2023optimal}
Sean O’Rourke, Van Vu, and Ke~Wang.
\newblock Matrices with gaussian noise: Optimal estimates for singular subspace
  perturbation.
\newblock \emph{IEEE Trans. Inf. Theory}, pages 1978--2002, 2024.
\newblock \doi{https://doi.org/10.1109/TIT.2023.3331010}.

\bibitem[Pavlidis and Horowitz(1974)]{pavlidis1974segmentation}
Theodosios Pavlidis and Steven~L Horowitz.
\newblock Segmentation of plane curves.
\newblock \emph{IEEE Trans. Comput.,}, 100\penalty0 (8):\penalty0 860--870,
  1974.
\newblock \doi{10.1109/T-C.1974.224041}.

\bibitem[Pollard(1981)]{pollard1981strong}
David Pollard.
\newblock Strong consistency of $k$-means clustering.
\newblock \emph{Ann. Stat.}, 9\penalty0 (1):\penalty0 135--140, 1981.
\newblock \doi{https://doi.org/10.1214/aos/1176345339}.

\bibitem[Rakhlin and Caponnetto(2006)]{rakhlin2006stability}
Alexander Rakhlin and Andrea Caponnetto.
\newblock Stability of $ k $-means clustering.
\newblock \emph{NeurIPS}, 19, 2006.
\newblock
  \doi{https://proceedings.neurips.cc/paper\_files/paper/2006/file/58191d2a914c6dae66371c9dcdc91b41-Paper.pdf}.

\bibitem[Rao(1964)]{rao1964use}
C~Radhakrishna Rao.
\newblock The use and interpretation of principal component analysis in applied
  research.
\newblock \emph{Sankhya A}, 26\penalty0 (4):\penalty0 329--358, 1964.
\newblock \doi{https://www.jstor.org/stable/25049339}.

\bibitem[Reiss and Wahl(2020)]{AveError}
Markus Reiss and Martin Wahl.
\newblock Nonasymptotic upper bounds for the reconstruction error of pca.
\newblock \emph{Ann. Stat.}, 48\penalty0 (2):\penalty0 1098--1123, 2020.
\newblock \doi{https://doi.org/10.1214/19-AOS1839}.

\bibitem[Rudelson and Vershynin(2013)]{rudelson2013hanson}
Mark Rudelson and Roman Vershynin.
\newblock Hanson-wright inequality and sub-gaussian concentration.
\newblock \emph{Electron. Commun. Probab.}, 18\penalty0 (92):\penalty0 1--9,
  2013.
\newblock \doi{https://doi.org/10.1214/ECP.v18-2865}.
\newblock URL \url{https://doi.org/10.1214/ECP.v18-2865}.

\bibitem[Seber and Lee(2003)]{seber2003linear}
George~AF Seber and Alan~J Lee.
\newblock \emph{Linear Regression Analysis}, volume 330.
\newblock John Wiley \& Sons, 2003.

\bibitem[Shalev-Shwartz and Ben-David(2014)]{shalev2014understanding}
Shai Shalev-Shwartz and Shai Ben-David.
\newblock \emph{Understanding Machine Learning: From Theory to Algorithms}.
\newblock Cambridge university press, 2014.

\bibitem[Shepard(1962)]{shepard1962analysis}
Roger~N Shepard.
\newblock The analysis of proximities: multidimensional scaling with an unknown
  distance function. i.
\newblock \emph{Psychometrika}, 27\penalty0 (2):\penalty0 125--140, 1962.
\newblock \doi{https://doi.org/10.1007/BF02289630}.

\bibitem[Singer(2006)]{singer2006graph}
Amit Singer.
\newblock From graph to manifold laplacian: The convergence rate.
\newblock \emph{Appl. Comput. Harmon. A.}, 21\penalty0 (1):\penalty0 128--134,
  2006.
\newblock \doi{https://doi.org/10.1016/j.acha.2006.03.004}.

\bibitem[Singh and Harrison(1985)]{singh1985standardized}
Ashbindu Singh and Andrew Harrison.
\newblock Standardized principal components.
\newblock \emph{Int. J. Remote Sens}, 6\penalty0 (6):\penalty0 883--896, 1985.
\newblock \doi{https://doi.org/10.1080/01431168508948511}.

\bibitem[Stewart and Sun(1990)]{stewart1990matrix}
Gilbert~W Stewart and Ji-guang Sun.
\newblock \emph{Matrix Perturbation Theory}.
\newblock Academic Press, 1990.

\bibitem[Stone(1961)]{stone1961approximation}
Henry Stone.
\newblock Approximation of curves by line segments.
\newblock \emph{Math. Comput.}, 15\penalty0 (73):\penalty0 40--47, 1961.
\newblock \doi{https://api.semanticscholar.org/CorpusID:121847482}.

\bibitem[Tao(2012)]{tao2023topics}
Terence Tao.
\newblock \emph{Topics in Random Matrix Theory}, volume 132.
\newblock American Mathematical Society, 2012.

\bibitem[Tong et~al.(2024)Tong, Wang, and Wang]{supp}
Xin~T. Tong, Wanjie Wang, and Yuguan Wang.
\newblock Supplementary material of "uniform error bound for pca matrix
  denoising".
\newblock 2024.

\bibitem[Tukey(1960)]{tukey1960survey}
John~W Tukey.
\newblock A survey of sampling from contaminated distributions. contributions
  to probability and statistics.
\newblock \emph{Essays in honor of Harold Hotelling}, 1960.

\bibitem[Vershynin(2010)]{vershynin}
Roman Vershynin.
\newblock Introduction to the non-asymptotic analysis of random matrices.
\newblock \emph{arXiv preprint arXiv:1011.3027}, 2010.
\newblock \doi{https://doi.org/10.48550/arXiv.1011.3027}.

\bibitem[Von~Luxburg(2007)]{von2007tutorial}
Ulrike Von~Luxburg.
\newblock A tutorial on spectral clustering.
\newblock \emph{Stat. Comput.}, 17:\penalty0 395--416, 2007.
\newblock \doi{https://doi.org/10.1007/s11222-007-9033-z}.

\bibitem[Von~Luxburg et~al.(2008)Von~Luxburg, Belkin, and
  Bousquet]{von2008consistency}
Ulrike Von~Luxburg, Mikhail Belkin, and Olivier Bousquet.
\newblock Consistency of spectral clustering.
\newblock \emph{Ann. Stat.}, 36\penalty0 (2):\penalty0 555--586, 2008.
\newblock \doi{10.1214/009053607000000640}.

\bibitem[Vu(2011)]{vu2011singular}
Van Vu.
\newblock Singular vectors under random perturbation.
\newblock \emph{Random Struct. Algor.}, 39\penalty0 (4):\penalty0 526--538,
  2011.
\newblock \doi{https://doi.org/10.1002/rsa.20367}.

\bibitem[Wahba et~al.(1999)]{wahba1999support}
Grace Wahba et~al.
\newblock Support vector machines, reproducing kernel hilbert spaces and the
  randomized gacv.
\newblock In \emph{Advances in Kernel Methods-Support Vector Learning},
  volume~6, pages 69--87. Citeseer, 1999.
\newblock \doi{10.7551/mitpress/1130.003.0009}.

\bibitem[Wedin(1972)]{wedin1972perturbation}
Per-{\AA}ke Wedin.
\newblock Perturbation bounds in connection with singular value decomposition.
\newblock \emph{BIT}, 12:\penalty0 99--111, 1972.
\newblock \doi{https://doi.org/10.1007/BF01932678}.

\bibitem[Yu et~al.(2015)Yu, Wang, and Samworth]{yu2015useful}
Yi~Yu, Tengyao Wang, and Richard~J Samworth.
\newblock A useful variant of the davis--kahan theorem for statisticians.
\newblock \emph{Biometrika}, 102\penalty0 (2):\penalty0 315--323, 2015.
\newblock \doi{https://doi.org/10.1093/biomet/asv008}.

\bibitem[Zhang et~al.(2022)Zhang, Cai, and Wu]{zhang2022heteroskedastic}
Anru~R Zhang, T~Tony Cai, and Yihong Wu.
\newblock Heteroskedastic pca: Algorithm, optimality, and applications.
\newblock \emph{Ann. Stat.}, 50\penalty0 (1):\penalty0 53--80, 2022.
\newblock \doi{https://doi.org/10.1214/21-AOS2074}.

\end{thebibliography}

\newpage
\appendix

\section{Detailed verification in Theorem \ref{thm:linftydavis}}
\begin{lemma}
Suppose the settings of Theorem \ref{thm:linftydavis} hold. 
If we fix $\lambda_r \geq c_X \sqrt{n}$ and $d\gtrsim n$,
\begin{align*}
&
\frac{\sigma(\sqrt{n}+\sqrt{d})}{\lambda_r}+\frac{\sigma n\sqrt{\log n}}{\lambda_r^2} +\frac{1}{\lambda_r^2}\Gamma(1+\sigma(\sqrt{d}+\sqrt{\log n}))\\
&\lesssim \sqrt{\frac{d}{n}}\sigma(1 +\frac{\sigma^3 d}{n}\sqrt{\log n} +\frac{\sigma^4 d}{n}\log n)+\sqrt{\log n}\sigma.
\end{align*}
Otherwise, if we fix $\lambda_r \geq c_X \sqrt{n}$ and $d\lesssim n$, 
\begin{align*}
\frac{\sigma}{\lambda_r^2} (\sqrt{nd}+n\sqrt{\log n}) +\frac{1}{\lambda_r^2}\Gamma(1+\sigma(\sqrt{d}+\sqrt{\log n}))\lesssim
\sqrt{\log n}\sigma(1+\sigma^4 \sqrt{\frac{d\log n}{n}}).
\end{align*}

\end{lemma}

\begin{proof}
Recall that  $\sqrt{n}\geq \lambda_r\geq 1+C_1\sigma(\sqrt{n}+\sqrt{d}),$
\[
\frac{d\sigma^2}{\lambda_r^2}\leq 
\frac{\sigma \sqrt{nd}}{C_1\lambda^2_r},\quad \frac{d\sigma^2}{\lambda_r^2}\leq 
\frac{\sigma\sqrt{d}}{C_1\lambda_r},\quad 
\frac{\sigma}{\lambda_r^2} \sqrt{n\log n}\leq \frac{\sigma}{\lambda_r^2} n\sqrt{\log n},\quad \frac{\sigma}{\lambda_r^2}\sigma n\leq \frac{\sigma}{\lambda_r^2} n\sqrt{\log n}.
\]
If we fix $\lambda_r \geq c_X \sqrt{n}$ and $d\gtrsim n$,
\begin{align*}
&
\frac{\sigma(\sqrt{n}+\sqrt{d})}{\lambda_r}+\frac{\sigma n\sqrt{\log n}}{\lambda_r^2} +\frac{1}{\lambda_r^2}\Gamma(1+\sigma(\sqrt{d}+\sqrt{\log n}))\\
&\lesssim 
\frac{\sqrt{d}\sigma }{\lambda_r}+\frac{\sigma n\sqrt{\log n}}{\lambda_r^2} +\frac{\sigma^2}{\lambda_r^4}\sqrt{\log n}(n\sqrt{d}+\sigma^2n\sqrt{nd}+\sigma^2 d\sqrt{nd})(1+\sigma\sqrt{\log n})+\frac{\Gamma}{\lambda_r^2}\\
&\lesssim 
\sqrt{\frac{d}{n}}\sigma+\sqrt{\log n}\sigma +\frac{\sigma^2}{n^2}\sqrt{\log n}(n\sqrt{d}+\sigma^2 d\sqrt{nd})(1+\sigma\sqrt{\log n})+\frac{\Gamma}{n}\\
&\lesssim
\sqrt{\frac{d}{n}}\sigma +\sqrt{\log n}\sigma+\frac{\sigma^2}{n}\sqrt{\log n}\sqrt{d}+\frac{\sigma^2}{n^2}\sqrt{nd\log n}\sigma^2 d+
\frac{\sigma^3}{n}\log n\sqrt{d}+\frac{\sigma^3}{n^2}\log n \sigma^2 d\sqrt{nd}+\frac{\Gamma}{n}\\
&\lesssim
\sqrt{\frac{d}{n}}\sigma(1 +\frac{\sigma}{\sqrt{n}}\sqrt{\log n}+\frac{\sigma^3 d}{n}\sqrt{\log n}+
\frac{\sigma^2}{\sqrt{n}}\log n +\frac{\sigma^4 d}{n}\log n )+\sqrt{\log n}\sigma+\frac{\Gamma}{n}\\
&\lesssim \sqrt{\frac{d}{n}}\sigma(1 +\frac{\sigma^3 d}{n}\sqrt{\log n} +\frac{\sigma^4 d}{n}\log n)+\sqrt{\log n}\sigma+\frac{\sigma}{n^2}\sqrt{\log n}(n+\sigma^2 d\sqrt{n})(1+\sigma\sqrt{\log n})\\
&\lesssim \sqrt{\frac{d}{n}}\sigma(1 +\frac{\sigma^3 d}{n}\sqrt{\log n} +\frac{\sigma^4 d}{n}\log n)+\sqrt{\log n}\sigma+\frac{\sigma}{n^2}\sqrt{\log n}(n+\sigma^2 d\sqrt{n}+\sigma\sqrt{\log n}+\sigma^3 d\sqrt{n\log n})\\
&\lesssim \sqrt{\frac{d}{n}}\sigma(1 +\frac{\sigma^3 d}{n}\sqrt{\log n} +\frac{\sigma^4 d}{n}\log n)+\sqrt{\log n}\sigma.
\end{align*}
Similarly, if we fix $\lambda_r \geq c_X \sqrt{n}$ and $d\lesssim n$, 
\begin{align*}
&\frac{\sigma}{\lambda_r^2} (\sqrt{nd}+n\sqrt{\log n}) +\frac{1}{\lambda_r^2}\Gamma(1+\sigma(\sqrt{d}+\sqrt{\log n}))\\
&\lesssim 
\sqrt{\log n}\sigma
+\frac{\sigma^2}{n^2}\sqrt{\log n}(n+\sigma^2 n\sqrt{n})(1+\sigma\sqrt{\log n})(\sqrt{d}+\sqrt{\log n})+\frac{\Gamma}{n}\\
&\lesssim
\sqrt{\log n}\sigma(1+
\frac{\sigma}{n}(1+\sigma^2\sqrt{n})(1+\sigma\sqrt{\log n})(\sqrt{d}+\sqrt{\log n}))\\
&= 
\sqrt{\log n}\sigma\bigg(1+
\frac{\sigma\sqrt{d}}{n}
+\frac{\sigma^3\sqrt{d}}{\sqrt{n}}
+\frac{\sigma^2\sqrt{d\log n}}{n}
+\frac{\sigma^4\sqrt{d\log n}}{\sqrt{n}}\\
&\quad \qquad+\frac{\sigma\sqrt{\log n}}{n}
+\frac{\sigma^3\sqrt{\log n}}{\sqrt{n}}
+\frac{\sigma^2\log n}{n}
+\frac{\sigma^4\log n}{\sqrt{n}}\bigg)+\frac{\Gamma}{n}\\
&\lesssim
\sqrt{\log n}\sigma(1+\sigma^4 \sqrt{\frac{d\log n}{n}})+\frac{\sigma}{n^2}\sqrt{\log n}(n+\sigma^2 n\sqrt{d}+\sigma\sqrt{\log n}+\sigma^3 n\sqrt{d\log n})\\
&\lesssim
\sqrt{\log n}\sigma(1+\sigma^4 \sqrt{\frac{d\log n}{n}}).
\end{align*}
% \sqrt{\frac{d}{n}}\sigma+
% \sqrt{\frac{\log n}{n}}\sigma+\sigma^2 +\frac{\sigma^2}{n^2}\sqrt{\log n}n\sqrt{d}+\frac{\sigma^2}{n^2}\sqrt{\log n}\sigma^2 n\sqrt{nd}+
% \frac{\sigma^3}{n^2}\log n n\sqrt{d}+\frac{\sigma^3}{n^2}\log n \sigma^2 n\sqrt{nd}\\
% &\lesssim
% \sqrt{\frac{d}{n}}\sigma+
% \sqrt{\frac{\log n}{n}}\sigma+\sigma^2+
% \sqrt{\frac{d}{n}}\sigma(\frac{\sigma}{\sqrt{n}}\sqrt{\log n}+\sigma^3\sqrt{\log n}+
% \frac{\sigma^2}{\sqrt{n}}\log n +\sigma^4\log n )

\end{proof}

\section{Proofs for zigzag line data objects}

\begin{lemma}
\label{lem:minsing}
Consider $A\in \reals^{l\times m},B\in \reals^{m\times n}$, Suppose $l\geq \max\{m,n\}$ then  
\[
\lambda_{\min \{m,n\}}(AB)\geq \lambda_{m}(A)\lambda_{\min \{m,n\}}(B).
\]
\end{lemma}
\begin{proof}
Denote the SVD of $B$ as $B=U_B \Lambda_B V_B^T$. We discuss the case $m \geq n$ and $m \leq n$ separately. 

Suppose $m\geq n$, then 
\[
\lambda_{n}(AB)=\sqrt{\lambda_n(B^TA^TAB)}
\geq \lambda_m(A)\sqrt{\lambda_n(B^TB)}=\lambda_m(A)\lambda_n(B).
\]
Suppose $m\leq n$, then let $U_B\in \reals^{m\times m}, \Lambda_B\in  \reals^{m\times m}$, and $V_B\in \reals^{n\times m}$. We see 
\[
\lambda_{m}(AB)
= \sqrt{\lambda_m(ABB^TA^T)}\geq \lambda_m(B)\sqrt{\lambda_m(AA^T)}=\lambda_m(A)\lambda_m(B).
\]
\end{proof}
\begin{proof}[Proof of Proposition \ref{prop:zigzag}]
According to the model, we rewrite $X^T_i = x_0^T+T^T_i W$, where $x_0 = \E[X_i]$ is a linear combination of $\{v_1, v_2, \cdots, v_R\}$, $T^T_i \in \reals^{1\times R} $, and $W \in \reals^{R\times d}$ with the $i$-th row being $v_i^T$. 
%In details, $D_v$ consists of the basis vectors of the space spanned by $\{v_1,\ldots, v_R\}$ and $S\in \reals^{R\times r}$ is defined so that the $i$-th row of $S D_v$ being $v^T_i$. 
%Hence, $D_v D_v^T = I_r$ and $SD_v$ gives the direction of the each segment. 
The vector $T_i$ is generated by $s_i \sim Unif(0,1)$, which can be further written as 
\[
T_{i,j}=\begin{cases}
t_{j+1} - t_{j},\quad &s_i \geq t_{j+1}; \\
s_i-t_j,\quad & t_j < s_i \leq t_{j+1};\\
0,\quad & s_i < t_{j}.
\end{cases}
\]
%where $s_i$ is a uniform $[0,1]$ random variable. 
Therefore, $\sum_{j}T_{i,j}=s_i \leq 1$.

Consider the largest eigenvalue first. Since all $v_i$'s have unit norm and $s_i \sim Unif(0,1)$, so $\|X_i\| \leq 1$. Hence, $\|X\| \leq \sqrt{n}$ by Theorem \ref{thm:eigenvalue}.

Now consider the smallest singular value. Again, by Theorem \ref{thm:eigenvalue}, we only need to show the $r$-th eigenvalue of $cov(X_i)$ is non-zero. 
Note that 
\[\text{cov}(X_i)=\text{cov}(W^TT_i)
=W^T\text{cov}(T_i) W. 
\]
Rewrite $W = SQ$, where $Q \in \reals^{r \times d}$ contains the basis of the $r$-dimensional subspace spanned by $\{v_1, v_2, \cdots, v_R\}$ and $S \in \reals^{R \times r}$ contains the linear coefficients so that each row of $SQ$ is $v_i^T$. Hence, $QQ^T = I_r$ and $S$ has rank $r$. 
By Lemma \ref{lem:minsing}  we have 
\[
\lambda_r(\text{cov}(X_i))\geq 
\lambda_R(\text{cov}(T_i)) \lambda_r^2(S)\lambda_r^2(Q) = \lambda_R(\text{cov}(T_i)) \lambda_r^2(S). 
\]
Since $W$ has rank $r$, the $r$-th singular value of $S$ is non-zero. Hence, it suffices to show $\lambda_R(\text{cov}(T_i))>0$. 

Now we investigate the $R$-th singular value of $\cov(T_i)$, where 
\[
\lambda_R(\cov(T_i))=\min\nolimits_{a\in \reals^R, \|a\|=1} \text{var}[T^T_ia].  
\]
Consider a vector $a$ with $\|a\|=1$. Then $T_i^Ta$ is  
\[
T^T_ia=(s_i-t_j)a_j+\sum_{k<j} (t_{k+1} - t_{k}) a_k,\quad t_j < s_i \leq t_{j+1} .
\]
Denote $m=\E[T^T_ia]$ and $b(j)=\sum_{k<j} (t_{k+1} - t_{k}) a_k-m$. Further define  
\begin{align*}
m_1(j)=\E[1_{s_i\in [t_j,t_{j+1}]}(s_i-t_j)]={(t_{j+1}-t_j)^2}/{2},\\
m_2(j)=\E[1_{s_i\in [t_j,t_{j+1}]}(s_i-t_j)^2]={(t_{j+1}-t_j)^3}/{3}.
\end{align*}
Introduce these terms into $\text{var}(T_i^Ta)$ and we have 
\begin{align*}
  \text{var}[T'_ia]& = \sum_j\E[1_{s_i\in [t_j,t_{j+1}]} [(s_i-t_j)a_j+b(j)]^2]\\
  &= \sum_j (b(j)^2+2b(j) m_1(j)a_j+ m_2(j) a^2_j)\\
  &\geq \sum_j (m_2(j)-(m_1(j))^2) a_j^2\\
  &\geq \min_{j\in [R]} [m_2(j)-(m_1(j))^2]\geq \frac{\rho^3}{3}-\frac{\rho^4}{4}.
\end{align*}
The first inequality is obtained by optimizing over all possible $b(j)$. The second inequality is obtained since $\sum_j |a_j|^2=1$. The last inequality is obtained because $f(x)={x^3}/3-{x^4}/{4}$ is increasing on $x\in [0,1]$. 

So, the result is proved by taking $c = \lambda_r(S)/2$. 
\end{proof}

\section{Proof for the statistical lower bound}
\begin{proof}[Proof of Theorem \ref{thm:lowerbound}] 
Consider any estimator $\Xhat$ based on $Z_{[n]}$ and denote it as $\Xhat(Z_{[n]})$. We first set up the bound for the noise level $\sigma$ and then discuss the sample size $n$. 

Conditional on $v$, $X_i$ and $Z_i$ are independent of $Z_{[n]/i}$. So we consider a new problem where both $Z_{[n]}$ and $v$ are known and the same estimator $\Xhat_i(Z_{[n]})$. By the Blackwell  thereom, the estimator can be improved by 
\[
\tilde{X}_i(v,Z_i)= \E[\Xhat_i(Z_{[n]})|v, Z_i]
\]
in the sense that 
\begin{equation}\label{eqn:allerror}
\E[\|\tilde{X}_i(v,Z_i)-X_i\|^2]\leq \E[\|\Xhat_i(Z)-X_i\|^2]\leq \epsilon^2.     
\end{equation}

Recall that the $\ell_2$ error of $\tilde{X}_i(v,Z_i)$ is larger than the Bayes error. We consider the Bayes error. Since $v$ is known, the estimation is $\Xhat_i = (\hat{t}_i+1) v$ and the uncertainty comes from $\hat{t}_i$ only. 
Let $P$ be the projection onto the direction of $v$ and $P_\bot$ the projection onto the complementary subspace. For notational simplicity, let $z_i = Z_i^T v/\|v\|$ and $e_i = \xi_i^T v/\|v\|$. The Bayes estimator of $t_i$ is given by 
\begin{align*}
\hat{t}_i&=\frac{\int^1_0 t\exp(-\tfrac{1}{2\sigma^2}\|Z_i-(t+1)v\|^2) dt}{\int^1_0 \exp(-\tfrac{1}{2\sigma^2}\|Z_i-(t+1)v\|^2) dt}\\
&=\frac{\int^1_0 t\exp(-\tfrac{1}{2\sigma^2}\|(z_i/\|v\|)v-(t+1)v\|^2) \exp(-\tfrac{1}{2\sigma^2}\|P_\bot Z_i\|^2)dt}{\int^1_0 \exp(-\tfrac{1}{2\sigma^2}\|(z_i/\|v\|)v-(t+1)v\|^2)\exp(-\tfrac{1}{2\sigma^2}\|P_\bot Z_i\|^2) dt}\\
&=\frac{\int^1_0 t\exp(-\tfrac{\|v\|^2}{2\sigma^2}|(t+1)-z_i/\|v\||^2) dt}{\int^1_0 \exp(-\tfrac{\|v\|^2}{2\sigma^2}|(t+1)-z_i/\|v\||^2) dt}
\\
&=\frac{\int^1_0 ((t+1)-z_i/\|v\|)\exp(-\tfrac{\|v\|^2}{2\sigma^2}|(t+1)-z_i/\|v\||^2) dt}{\int^1_0 \exp(-\tfrac{\|v\|^2}{2\sigma^2}|(t+1)-z_i/\|v\||^2) dt}
+\frac{z_i}{\|v\|}-1\\
&=\frac{z_i}{\|v\|}-1+\frac{\sigma^2}{\|v\|^2}\frac{\phi(\sigma^{-1}(2\|v\|-z_i))-\phi(\sigma^{-1}(\|v\|-z_i))}{\Phi(\sigma^{-1}(2\|v\|-z_i))-\Phi(\sigma^{-1}(\|v\|-z_i))}\\
&=t_i + \frac{e_i}{\|v\|}+\frac{\sigma^2}{\|v\|^2}\frac{\phi(\sigma^{-1}(2\|v\|-z_i))-\phi(\sigma^{-1}(\|v\|-z_i))}{\Phi(\sigma^{-1}(2\|v\|-z_i))-\Phi(\sigma^{-1}(\|v\|-z_i))}
\end{align*}
The estimation error $|\hat{t}_i - t_i|$ consists of $e_i/\|v\|$ and a complicated term. We now bound it. 
Consider a set $B=\{(t_i,e_i): e_i (t_i-0.5)>0, |e_i|>\sigma\}$, then $P(B) = \Phi(-1)/2$, where $\Phi$ is the CDF of standard Gaussian.  If $t_i>0.5$, $e_i>0$, then $z_i>1.5\|v\|$ and so $\phi(\sigma^{-1}(2\|v\|-z_i))-\phi(\sigma^{-1}(\|v\|-z_i))>0$. It suggests that  
\begin{equation}\label{eqn:bayes1}
|\hat{t}_i-t_i|\geq {|e_i|}/{\|v\|}>{\sigma}/{\|v\|}.     
\end{equation}
With a similar analysis, \eqref{eqn:bayes1} still holds when $t_i<0.5$ and $\xi_i<0$. 
As a conclusion, $|\hat{t}_i-t_i| \geq \sigma/\|v\|$ at the occurrence of $B$.

Further if  $\|v\| \leq 2$, which happens with probability at least $3/4$ by Markov inequality, then 
\[
|\hat{t}_i-t_i|\geq \sigma/2,
\]
and the error follows 
\[
\E[\|\hat{t}_i-t_i\|^2] \geq \E[\|\hat{t}_i-t_i\|^21_{B}]\geq \frac18\sigma^2P(B) = \frac{\Phi(-1)}{16}\sigma^2. 
\] 
Combine it with \eqref{eqn:allerror}, we need $\sigma \leq 4\epsilon/\sqrt{\Phi(-1)}$.

Next we consider the sample size $n$. Here we consider a new problem where both $t_{[n]}$ and $Z_{[n]}$ are observed but $v$ is unknown. Then for any estimator $\Xhat_i(Z_{[n]})$, we can design an estimator of $v$ as
\[
\hat{v}=\frac{1}{n}\sum_{i=1}^n  \frac{\Xhat_i(Z_{[n]})}{t_i+1}. 
\]
Then the error is bounded by 
\[
\E \|\vhat-v\|^2 =  
\E \left\|\frac{1}{n}\sum_{i=1}^n\frac{\Xhat_i(Z_{[n]})-X_i}{t_i+1}\right\|^2\leq \E[ \max_{i} \|X_i-\Xhat_i(Z_{[n]})\|^2]\leq \epsilon^2. 
\]
Meanwhile, we know $l^2$ error of $\hat{v}_i(t_{[n]},Z_{[n]})$ will be larger than the Bayes error of the Bayes estimator for $v$. The prior of $v$ is $\mathcal{N}(0, \frac{1}{d}I_d)$ and the data $Z_{[n]}$ follows $Z_i \sim N((t_i+1)v, \sigma^2 I_d)$ with given $t_{[n]}$. Therefore, the posterior distribution of $v$ is also Gaussian, with mean and covariance
\[
\tilde{v}=\frac{\sum_{i=1}^n(t_i+1)Z_i}{
\sigma^2d+\sum_{i=1}^n(t_i+1)^2},\quad \tilde{C}=(\frac{1}{d}+\frac{1}{\sigma^2}\sum_{i=1}^n(t_i+1)^2)^{-1}I_d. 
\]
The Bayes error is 
\[
\E[\|\tilde{v}-v\|^2]=\E[\text{tr}(\tilde{C})]\geq \frac{d\sigma^2}{4n+\sigma^2/d}.
\]
Combine it with that $\E[\|\tilde{v}-v\|^2] \leq \epsilon^2$, we need ${d\sigma^2}/({4n+\sigma^2/d}) \leq \epsilon^2$.
Since we assume $\epsilon^2\leq \frac1{16} d^2$, we have $\sigma^2/d\leq \frac{4}{15}n$, which further have $\E[[\|\tilde{v}-v\|^2]]\geq \frac{d\sigma^2}{5n^2}$. So $n<\frac{d\sigma^2}{5\epsilon^2}$ is impossible.
\end{proof}

\section{Denoising applications}
\subsection{Empirical risk minimization for supervised learning}\label{sec:erm}
Empirical risk minimization (ERM) is a widely used approach for training statistical models \cite{hastie2009elements,shalev2014understanding}. In ERM, the dataset consists of pairs $(X_i, y_i)$, where $X_i\in \mathbb{R}^d$ and $y_i$ is typically a scalar. The prediction error is measured by a loss function $F(X_i, y_i, \theta)$, where the specific definition of $F$ varies depending on the model. 
The goal of ERM is to find the optimal model parameters $\theta^*$ that minimize the empirical prediction error as follows. 
\begin{equation}
\label{eqn:ERMclean}
\theta^*=\arg\min_{\theta} f(\theta), \quad f(\theta):=\frac1n\sum_{i=1}^n F(X_i,y_i,\theta). 
\end{equation}

Some well-known examples are listed below: 
\begin{itemize}
    \item In standard linear regression \cite{seber2003linear}, $y_i \in \reals$, $\theta \in \reals^d$ and $F$ is given by 
    \[
    F(X_i,y_i,\theta)=\frac12(\theta^TX_i-y_i)^2.
    \]
    \item In linear regresion with Tuckey's biweight loss function \cite{tukey1960survey}, 
    \[
    F(X_i,y_i,\theta)=\rho(\theta^TX_i-y_i),\quad \text{where } \rho(u)=
    \frac{c^2}{6}(1-(1-(u/c)^2)^3 1_{|u|<c}).
    \]
    Here $c$ is some threshold constant. The introduction of the $\rho$ function is to reduce the influence of possible outliers from data. 
    \item In logistic regression \cite{menard2002applied}, $y_i\in \{+1,-1\}$ and the loss function $F$ is given by 
    \[
    F(X_i,y_i,\theta)= \log(1+\exp(-y_iX_i^T\theta)).
    \] 
    \item In reproducing kernel Hilbert space regression \cite{wahba1999support}, we have a reproducing kernel $k(\cdot, \cdot)$ and its corresponding Hilbert space $\mathcal{H}$. The goal is to find a function $f\in \mathcal{H}$ to minimize $L(f)=\frac1n \sum_{i=1}^n (f(x_i) - y_i)^2$.   The kernel regression can also be viewed as a standard ERM problem. To do it, we parameterize $f_\theta(x)=\langle\theta,\varphi(x)\rangle$, where $\varphi(x)$ is the feature map and the kernel function $k(x,y)=\langle \varphi(x),\varphi(y)\rangle$. The loss function $F$ is
    \[
    F(X_i,y_i,\theta)=\frac12(\langle\theta,\varphi(X_i)\rangle-y_i)^2.
    \]
     \item In an $m$-layered neural networks, we denote $\theta_i \in \reals^d$ as the loading of the $i$-th layer. Hence, the parameters matrix is $\theta = [\theta_1,\ldots,\theta_m] \in \reals^{m \times d}$. The loss is given by 
    \[
    F(X_i,y_i,\theta)= (y_i- \sigma(\theta_m ^T\cdots\sigma(\theta^T_2\sigma(\theta^T_1 X_i))\cdots))^2,
    \] 
    where the function $\sigma(\cdot)$ can be taken as various nonlinear functions,  such as ReLu and sigmoid function.  
\end{itemize}

In many applications, we do not have access to the clean data $X_i$, but only the noisy data $Z_i = X_i + \xi_i$. The high noise in $Z_i$ will cause large error in the estimation of $\theta$, if we use $Z_i$ instead of $X_i$ in $F$ directly. 
It's natural to guess that the denoised data $\Xhat_i$ will lead to an estimator reasonably close to $\theta^*$. Let $\thetahat^*$ be the estimator, where 
\begin{equation}
\label{eqn:ERMnoisy}
\hat{\theta}^*=\arg\min \hat{f}(\theta), \quad \hat{f}(\theta):=\frac1n\sum_{i=1}^n F(\Xhat_i,y_i,\theta). 
\end{equation}
It is natural to ask how would $\thetahat^*$ performs under the empirical risk function $f$ in \eqref{eqn:ERMclean}, and how close would $\thetahat^*$ be from $\theta^*$. To answer these questions, we made the following assumptions.
\begin{aspt}\label{aspt:lip}
    The loss function $F(x,y_i,\theta)$ is globally Lipschitz in $x$. That is, there are constants $L>0$ and $\Delta>0$,  so that when $\|z-x\|\leq \Delta$, there is 
    \[
    |F(z,y_i,\theta)-F(x,y_i,\theta)|\leq L\|z-x\|.
    \]
\end{aspt}
 In general Assumption \ref{aspt:lip} is easy to verify when the data are generated in bounded domains, even for complex nonlinear models such as neural networks. But it alone cannot guarantee $\theta^*$ is unique, so if we want to infer parameter error, we need some additional conditions. 
\begin{aspt}\label{aspt:mini}
    The empirical loss function $f$ in \eqref{eqn:ERMclean} has a unique  minimizer $\theta^*$ and it is also a local minimum. That is, there are  constants $\gamma, \delta, \lambda_f>0$, so that $f(\theta)-f(\theta^*)\geq \gamma$ when $\|\theta-\theta^*\|\geq \delta$ and $\nabla^2 f(\theta)\succeq \lambda_f I$ when $\|\theta-\theta^*\|\leq \delta$. 
\end{aspt}
Suppose $f$ is strongly convex, Assumption \ref{aspt:mini} holds immediately. This is also the most well understood ERM regime. Assumption \ref{aspt:mini} also allows for general non-convex problems \cite{shalev2014understanding}. Similar version of it can be found in machine learning literature where finding $\theta^*$ is of interest \cite{dong2021replica}. 
\begin{prop}
\label{prop:ERM}
Consider an ERM problem where Assumption \ref{aspt:lip} holds.
Suppose $\err{\Xhat - X} \leq \epsilon$  where $L\epsilon <\Delta.$ Let $\theta^*$ and $\thetahat^*$ be the minimizers of $f$ and $\fhat$, respectively. Then $f(\thetahat^*)$ is close to the minimum $f(\theta^*)$, in the sense that 
\[
|f(\thetahat^*) - f(\theta^*)| \leq 2L\epsilon.
\]
Suppose Assumption \ref{aspt:mini} holds in addition with $L\epsilon < \gamma$, then the 
 two minimizers are close, 
    \[
\|\thetahat^*-\theta^*\|^2\leq {2L\epsilon}/{\lambda_f}.
\]
\end{prop}
Proposition \ref{prop:ERM} indicates that the ERM training result using the PCA denoised data is as good the training result using clean data, assuming Assumption \ref{aspt:lip}. If Assumption \ref{aspt:mini} is also in place, then the learned parameters will also be close to each other. The proof can be found in the \supp.

\begin{proof}
For the first claim, simply note that, 
\[
|\fhat(\theta)-f(\theta)|\leq \frac1n\sum_{i=1}^n|F(\Xhat_i,y_i,\theta)-F(X_i,y_i,\theta)|
\leq L\epsilon. 
\]
Note that $f(\theta^*) \leq f(\thetahat^*)$ and $\fhat(\thetahat^*) \leq \fhat(\theta^*)$, and we have 
\begin{eqnarray*}
f(\thetahat^*) & \leq & \fhat(\thetahat^*) + | \fhat(\thetahat^*) - f(\thetahat^*) | \leq \fhat(\thetahat^*) + L\epsilon 
\leq \fhat(\theta^*) + L\epsilon \\
& \leq & f(\theta^*) + |\fhat(\theta^*) - f(\theta^*)| + L\epsilon \leq f(\theta^*)+2L\epsilon.    
\end{eqnarray*}
This leads to our first claim. 

For the second claim, note that by $|f(\thetahat^*) - f(\theta^*)| \leq 2L\epsilon < \gamma$, we have $\|\theta^*-\theta\|\leq \delta$. Then by Taylor expansion, we have 
\[
f(\thetahat^*)=f(\theta^*)+\frac12(\thetahat^*-\theta^*)^T\nabla^2 f(\tilde{\theta}) (\thetahat^*-\theta^*)\geq f(\theta^*)+\frac12 \lambda_{f} \|\thetahat^*-\theta^*\|^2. 
\]
Combining it with $|f(\thetahat^*) - f(\theta^*)| \leq 2L\epsilon$ again, it leads to 
\[
\|\thetahat^*-\theta^*\|^2\leq {4L\epsilon}/{\lambda_f}.
\]
\end{proof}
\subsection{Proofs for other applications}
\begin{proof}[Proof of Proposition \ref{cor:cluster}]
For the first claim, simply note that, 
\begin{align*}
|\fhat(\calC, m)-f(\calC, m)|
&\leq \frac1n\sum_{k\in [K]} \sum_{i\in \calC_k} |\|X_i-m_k\|^2-\|\Xhat_i-m_k\|^2|\\
&\leq \frac1n\sum_{k\in [K]} \sum_{i\in \calC_k} (2\delta\epsilon + \epsilon^2)\leq 2\epsilon. 
\end{align*}
The last inequality comes from the fact $\delta \leq 1$ because $\|X_i\|\leq 1$ for all $i \in [n]$ and $\epsilon = o(1)$. 

Recall that $\|X_i\| \leq 1 + \epsilon$ and so likewise $|\fhat(\calChat, \mhat)-f(\calChat, \mhat)|\leq 2\epsilon$. This further leads to
\[
f(\calChat, \mhat) \leq \fhat(\calChat, \mhat)+2\epsilon\leq \fhat(\calC, m)+2\epsilon\leq f(\calC, m)+4\epsilon.
\]
Consider the second claim. Note that given any division $\tilde{\calC}$, then the minimizer $m$ of $f(\tilde\calC, m)$ and $\fhat(\tilde\calC, m)$ can be found as 
\[
m_k(\tilde{\calC})=\frac{\sum_{i\in \tilde{\calC}_k} X_i}{|\tilde{\calC}_k|},\quad 
\mhat_k(\tilde{\calC})=\frac{\sum_{i\in \tilde{\calC}_k} \Xhat_i}{|\tilde{\calC}_k|}.
\]
For the special case $\tilde{\calC} = \calChat$, we have $\mhat_k = \mhat_k(\calChat)$. 

Now we consider $\fhat(\calChat, \mhat)$. According to the definition,  
\begin{align*}
\fhat(\calChat, \mhat)&\leq \fhat(\calC, \mhat(\calC))\leq \fhat(\calC, m) = \frac1n\sum_{k\in[K]} \sum_{i\in \calC_k} \|\Xhat_i-m_k\|^2\\
&\leq \frac2n\sum_{k\in[K]} \sum_{i\in \calC_k} (\|\Xhat_i-X_i\|^2+\|X_i-m_k\|^2)\\
&\leq 2\delta^2+2\epsilon^2. 
\end{align*}
% denote 
% \[
% m_k=\frac{\sum_{i\in \calC_k} X_i}{|\calC_k|},\quad 
% \mhat_k=\frac{\sum_{i\in \calChat_k} X_i}{|\calChat_k|}.
% \]
Now we consider the lower bound. 
\begin{align}
\label{tmp:Llower}
\fhat(\calChat, \mhat)&=\frac1n\sum_{k\in[K]}\sum_{i\in \calC_k} \| \Xhat_i -\mhat_{j:i \in \calChat_j}\|^2\nonumber\\
&\geq  \frac1n\sum_{k}\sum_{i\in \calC_k} (-\| \Xhat_i -m_k\|^2+\frac12\|m_k-\mhat_{j:i \in \calChat_j}\|^2)\nonumber\\
& =  -\fhat(\calC, m) + \frac{1}{2n}\sum_{k}\sum_{i\in \calC_k} \|m_k-\mhat_{j:i \in \calChat_j}\|^2.
\end{align}
Combining the two bound above and we have 
\[
\frac1n\sum_{k}\sum_{i\in \calC_k} \|m_k-\mhat_{j:i \in \calChat_j}\|^2 \leq 2(\fhat(\calC, m) + \fhat(\calChat, \mhat)) \leq 8(\delta^2 + \epsilon^2).
\]

Now we compare $\calChat$ and $\calC$. To simplify the notations, we use $\ell(i)$ to denote the true label of node $i$ so that $i \in \calC_{\ell(i)}$ and use $\hat{\ell}(i)$ to denote the estimated label that $i \in \calChat_{\hat{\ell}(i)}$. We need a projection $\pi:[K]\to[K]$ so that $\ell$ and $\pi(\hat{\ell})$ will match. The projection $\pi$ is defined by matching the nearest centers: 
\begin{equation}\label{eqn:pi}
    \pi(k)=\arg\min_{j} \|m_j-\mhat_{k}\|, \quad k \in [K].
\end{equation}
Then the overall distances between centers are 
\begin{align}
 c_0\|m_k-\mhat_{\pi(k)}\|^2 &\leq \frac1n\sum_{k}\sum_{i\in \calC_k}\|m_k-\mhat_{\pi(k)}\|^2 \nonumber\\
 &\leq \frac1n\sum_{k}\sum_{i\in \calC_k}\|m_k-\mhat_{\hat{\ell}(i)}\|^2\leq 8(\delta^2+\epsilon^2). 
\end{align}

Given $\pi(\hat{\ell})$, we want to show $\pi(\hat{\ell}) = \ell$. To prove it, we suppose there is a data point $i$ where $\pi(\hat\ell(i))\neq \ell(i) = k$. Define a new label vector $\ell'$ that differs from $\hat{\ell}$ on the data point $i$ only, where $\ell'(i) = \pi^{-1}(k)$. Correspondingly we have $\calC'$. We want to show that $\calC'$ will result in a smaller $\fhat$ than $\fhat(\calChat, \mhat)$, and hence $\calChat$ will not be a solution of $k$-means.
\begin{align*}
\fhat(\calC', \mhat(\calC'))-\fhat(\calChat, \mhat)&\leq 
\fhat(\calC', \mhat)-\fhat(\calChat, \mhat)\\
&=
\frac1n\sum_{j=1}^n \|\Xhat_j-\mhat_{\ell'(j)}\|^2-\frac1n\sum_{j=1}^n \|\Xhat_j-\mhat_{\hat{\ell}(j)}\|^2\\
&=\frac1n\|\Xhat_{i}-\mhat_{{\ell'}(i)}\|^2-
\frac1n\|\Xhat_{i}-\mhat_{\hat{\ell}(i)}\|^2,
\end{align*}
where the last equality comes from that $\hat{\ell}(j) = \ell'(j)$ for all $j \neq i$. Further, 
\begin{align*}
\|\Xhat_{i}-\mhat_{\ell'(i)}\|&\leq 
\|\Xhat_{i}-m_{\ell(i)}\|+\|m_{\ell(i)}-\mhat_{\pi(\ell(i))}\|\\
&\leq \delta+\epsilon+2\sqrt{2(\delta^2+\epsilon^2)/c_0}.
\end{align*}
Recall that we assume $\pi(\hat{\ell}(i))\neq \ell(i) = k$. Denote $j=\pi(\hat{\ell}(i))$, then 
\begin{align*}
\|\Xhat_{i}-\mhat_{\hat{\ell}(i)}\|&\geq \|m_{k}-m_{j}\|-( 
\|\Xhat_{i}-m_{k}\|+\|m_{j}-\mhat_{\hat{\ell}(i)}\|)\\
&\geq c_m-(\delta+\epsilon+2\sqrt{2(\delta^2+\epsilon^2)/c_0})>\|\Xhat_{i}-\mhat_{\ell'(i)}\|.
\end{align*}
So $\fhat(\calC', \mhat(\calC')) < \fhat(\calChat, \mhat)$ and $\ell'$ will be a strictly better solution to the $k$-means objective. This contradicts the definition of $\hat{\ell}$. So for all data points, $\pi(\hat\ell) = \ell$.     
\end{proof}

\begin{proof}[Proof of Proposition \ref{prop:lap}]
Consider $\|\Lhat - L\|_{\infty}$ first. Since $\err{X - \Xhat} \leq \epsilon$ and $k(x, y)$ is $l$-Lipschitz in $x$ and $y$, so there is a constant $C_3=C_3(l)> 0$ so that 
\begin{align*}
&|(k(X_i,X_j)-k(\Xhat_i,\Xhat_j))|\\
&\leq
|(k(X_i,X_j)-k(\Xhat_i,X_j))|+|(k(\Xhat_i,X_j)-k(\Xhat_i,\Xhat_j))|\leq C_3\epsilon,
\end{align*}
and 
\[
|D_{ii}-\hat{D}_{i,i}|=\biggl|\sum_{j \in [n]} (k(X_i,X_j)-k(\Xhat_i,\Xhat_j))\biggr|\leq C_3n\epsilon. 
\]
Hence, for the normalized Laplacian, there is a $C_4=C_4(K,l)$
\[
|\Lhat_{i,j}-L_{i,j}|=|\frac{k(X_i,X_j)}{\sqrt{D_{i,i}D_{j,j}}}-\frac{k(\Xhat_i,\Xhat_j)}{\sqrt{\Dhat_{i,i}\Dhat_{j,j}}}|\leq \frac{C_4\epsilon}{n}. 
\]
Therefore, the matrix $\ell_{\infty}$ and $\ell_{1}$ norm follows that 
\[
\|\Lhat-L\|_1=\|\Lhat-L\|_\infty = \max_i\sum_{j}|\Lhat_{i,j}-L_{i,j}|\leq C_4\epsilon. 
\]
The first identity comes from the fact that $L,\Lhat$ are symmetric.
We further present a result about $\|\Lhat\|_{\infty}$ here. Since $\Dhat_{i,i} \geq n/K$ and $\Ahat_{i,j} \leq K$, 
\begin{equation}\label{eqn:lapinf}
\|\Lhat\|_\infty = \max_i\sum_{j}|\Lhat_{i,j}|\leq 
\max_i\{1+\sum_j\frac{ \Ahat_{i,j}}{\sqrt{\hat{D}(i,i)\hat{D}(j,j)}}\}\leq 1 + K^2.
\end{equation}

Consider the eigenvalues and eigenvectors. 
Since $\lambda_0$ is simple, it has a neighborhood $B$ of radius $\delta$ so that its intersection with the spectrum of $\mathcal{L}$ is $\{\lambda\}$.
Theorem 15 of \cite{von2008consistency} demonstrates that there is an $N_0$, so that when $n>N_0$, the intersection of $B$ with spectrum of $L$ is $\{\lambda\}$. In other words, $\lambda$ is a simple eigenvalue for $L$.
 
By Weyl's inequality, $|\lambda_i(\Lhat) - \lambda_i(L)| \leq \|\Lhat - L\| \leq C_4\epsilon$, where the last inequality comes from $\|\Lhat - L\| \leq \sqrt{\|\Lhat - L\|_1\|\Lhat - L\|_{\infty}} \leq C_4\epsilon$.
Therefore, $\Lhat$ has an eigenvalue $\hat{\lambda}$  so that 
\[
|\lambda-\hat{\lambda}|\leq C_4\epsilon.
\]
Let $\hat{P} v$ be the eigen-projection of $v$ onto the one-dimensional space spanned by $\hat{v}$ in $l_\infty$ norm, i.e. 
\[
\hat{P} v=\hat{c} \hat{v},\quad \text{where } \hat{c}=\arg\min_c \|v-c \hat{v}\|_\infty.
\]
Theorem 7 of \cite{von2008consistency} demonstrates that for sufficiently  small $\epsilon$, there is a constant $C$ independent of $d$ and $n$
\[
\|v-\hat{P}v\|_\infty\leq C(\|(L-\Lhat)v\|_\infty+\|v\|_\infty\|(L-\Lhat)\Lhat\|_\infty)
\leq C\epsilon\|v\|_\infty.
\]
Meanwhile, Proposition 18 of \cite{von2008consistency} shows that there is an $a\in \{\pm 1\}$ so that
\[
\|v-a\hat{v}\|_\infty\leq 2\|v-\hat{P}v\|_\infty
\leq 2C\epsilon \|v\|_\infty. 
\]
So the second claim is proved.
\end{proof}

\end{document}